\documentclass[12pt,twoside]{article}
\usepackage{amsfonts,amssymb,amsmath} 
\usepackage[T1]{fontenc}
\usepackage{graphicx,color}
\usepackage{geometry, indentfirst}

\usepackage[latin1]{inputenc}

\newtheorem{theorem}{Theorem}[section]
\newtheorem{lemma}[theorem]{Lemma}
\newtheorem{proposition}[theorem]{Proposition}
\newtheorem{corollary}[theorem]{Corollary}

\newtheorem{remark}[theorem]{Remark}

\newenvironment{proof}[1][Proof]{\noindent\textbf{#1.} }{\hfill \rule{0.5em}{0.5em}}

\begin{document}
\title{ Multi-bump solutions for a class of quasilinear problems involving variable exponents \footnote{Partially supported by INCT-MAT and PROCAD}}

\author{Claudianor O. Alves\footnote{C.O. Alves was partially supported by CNPq/Brazil
303080/2009-4, e-mail:coalves@dme.ufcg.edu.br}
\,\,\, and \,\,\,  Marcelo C. Ferreira \footnote{e-mail:marcelo@dme.ufcg.edu.br}\\
Universidade Federal de Campina Grande\\
Unidade Acad\^emica de Matem\'atica\\
CEP:58429-900, Campina Grande - PB, Brazil.}

\date{}

\maketitle



\begin{abstract}
   We establish the existence of multi-bump solutions for the following class of quasilinear problems
$$
   - \Delta_{ p(x) } u + \big( \lambda V(x) + Z(x) \big) u ^{ p(x)-1 }  = f(x,u) \text{ in } \mathbb R^N, \, u \ge 0 \text{ in } \mathbb R^N,
$$
where the nonlinearity $ f \colon \mathbb R^N \times \mathbb R \to \mathbb R $ is a continuous function having a subcritical growth and potentials $ V, Z \colon \mathbb R^N \to \mathbb R $ are continuous functions verifying some hypotheses. The main tool used is the variational method.
\end{abstract}


\section{Introduction}


In this paper, we considered the existence and multiplicity of solutions for the following class of problems
$$
   \big( P_\lambda \big) \;
	 \begin{cases}
		  - \Delta_{ p(x) } u + \big( \lambda V(x) + Z(x) \big) u^{ p(x)-1 }= f(x,u), \text{ in } \mathbb R^N, \\
		  u \ge 0, \text{ in } \mathbb R^N, \\
			u \in W^{ 1,p(x) } \big( \mathbb R^N \big),
		\end{cases}
$$
where $ \Delta_{ p(x) } $ is the $ p(x) $-Laplacian operator given by
$$
   \Delta_{ p(x) } u = \text{div} \left( \big| \nabla u \big|^{p(x)-2} \nabla u \right).
$$
Here, $ \lambda > 0 $ is a parameter, $ p \colon \mathbb R^N \to \mathbb R $ is a Lipschitz function, $ V,Z \colon \mathbb R^N \to \mathbb R $ are continuous functions with $ V \ge 0 $, and  $ f \colon \mathbb R^N \times \mathbb R \to \mathbb R $ is continuous having a subcritical growth. Furthermore, we take into account the following set of hypotheses:	

\begin{enumerate}
\item[($H_1$)] $ 1 < p_- \le p_+ < N $.
\item[($H_2$)] $ \Omega = \text{int } V^{ -1 } (0) \ne \emptyset $ and bounded, $ \overline{\Omega} = V^{ -1 }(0) $ and  $ \Omega $ can be decomposed in $ k $ connected components $ \Omega_1, \ldots, \Omega_k $ with $ \text{dist} \big( \Omega_i, \Omega_j \big) > 0, \, i \ne j $.
\item[($H_3$)] There exists $ M > 0 $ such that
$$
\lambda V(x) + Z(x) \ge M, \, \forall x \in \mathbb R^N, \lambda \ge 1.
$$
\item[($H_4$)] There exists $ K > 0 $ such that
$$
\big| Z(x) \big| \le K, \, \forall x \in \mathbb R^N.
$$
\item[($f_1$)] $$
\limsup_{ |t| \to \infty } \frac{|f(x,t)|}{|t|^{ q(x)-1 }} < \infty, \text{ uniformly in } x \in \mathbb R^N,
$$
where $ q \colon \mathbb R^N \to \mathbb R $ is continuous with $ p_+ < q_- $ and $ q \ll p^* $. 	
\item[($f_2$)] $ f(x,t) = o \big( |t|^{ p_+ - 1} \big), \, t \to 0, \text{ uniformly in } x \in \mathbb R^N $.
\item[($f_3$)] There exists $ \theta > p_+ $ such that
$$
0 < \theta F(x,t) \le f(x,t) t, \, \forall x \in \mathbb R^N, t > 0,
$$
where $ F(x,t) = \int_0^t f(x,s) \, ds $.
\item[($f_4$)] $ \dfrac{f(x,t)}{t^{p_+ - 1}} $ is strictly increasing in $ (0,\infty) $, for each $ x \in \mathbb R^N $. 
\item[($f_5$)] $ \forall a, b \in \mathbb R, \, a < b,  \, \displaystyle \sup_{ x \in \mathbb R^N \atop t \in [a,b] } |f(x,t)| < \infty $.						
\end{enumerate}

\vspace{0.3 cm}

A typical example of nonlinearity verifying $ (f_1)-(f_5) $  is
$$
   f(x,t) = |t|^{ q(x)-2 }t, \, \forall \, x \in \mathbb R^N \, \mbox{ and } \, \forall t \in \mathbb R,
$$
where $ p_+ < q_- $ and $ q \ll p^* $.

\vspace{0.3 cm}
Partial differential equations involving the $ p(x) $-Laplacian arise, for instance, as a mathematical model for problems involving electrorheological fluids and image restorations, see \cite{Acerbi1,Acerbi2,Antontsev,CLions,Chen,Ru}. This explains the intense research on this subject in the last decades. A lot of works, mainly treating nonlinearities with subcritical growth, are available (see \cite{Alves2,Alves3,AlvesBarreiro,AlvesSouto,AlvesFerreira,AlvesFerreira1,Fan1,Fan,FanZhao0,FZ1,FanShenZhao,BSS,FuZa,MR} for interesting works). Nevertheless, to the best of the author's knowledge, this is the first work dealing with multi-bump solutions for this class of problems.

The motivation to investigate problem $ \big( P_\lambda \big) $ in the setting of variable exponents has been the papers \cite{Alves} and \cite{DingTanaka}. In \cite{DingTanaka},  inspired by \cite{DelPinoFelmer} and \cite{Sere} the authors considered $ \big( P_\lambda \big) $ for $ p = 2 $ and $ f(u) = u^q $, $ q \in \big(1, \frac{N+2}{N-2} \big) $ if $ N \ge 3 $; $ q \in (1, \infty) $ if $ N = 1, 2 $. The authors showed that $ \big( P_\lambda \big) $ has at least $ 2^k-1 $ solutions $u_\lambda$ for large values of $ \lambda $. More precisely, one solution for each non-empty subset $ \Upsilon $ of $ \{ 1,\ldots,k \} $. Moreover, fixed $ \Upsilon \subset \{ 1,\ldots,k \}$, it was proved that, for any sequence $ \lambda_n \to \infty $ we can extract a subsequence $( \lambda_{n_i}) $ such that $( u_{ \lambda_{n_i} } )$ converges strongly in $ W^{ 1,p(x) } \big( \mathbb R^N \big) $ to a function $ u $, which satisfies $ u = 0 $ outside $ \Omega_\Upsilon = \bigcup_{ j \in \Upsilon } \Omega_j $ and $ u_{|_{\Omega_j}}, \, j \in \Upsilon $, is a least energy solution for
$$
   \begin{cases}
		  - \Delta u + Z(x) u  = u^q, \text{ in } \Omega_j, \\
		  u \in H^1_0 \big( \Omega_j \big), \, u > 0, \text{ in } \Omega_j.
	 \end{cases}
$$

In \cite{Alves}, employing some different arguments than those used in \cite{DingTanaka}, Alves extended the results described above to the $ p$-Laplacian operator, assuming that in $ \big( P_\lambda \big) $ the nonlinearity $ f $ possesses a subcritical growth and $ 2 \le p < N $. In particular, fixed $ \Upsilon \subset \{ 1,\ldots,k \}$, for any sequence $ \lambda_n \to \infty $ we can extract a subsequence $ (\lambda_{n_i}) $ such that $ (u_{ \lambda_{n_i} }) $ converges strongly in $ W^{ 1,p(x) } \big( \mathbb R^N \big) $ to a function $ u $, which satisfies $ u = 0 $ outside $ \Omega_\Upsilon $ and $ u_{|_{\Omega_j}}, \, j \in \Upsilon $, is a least energy solution for
$$
   \begin{cases}
		  - \Delta_p u + Z(x) u  = f(u), \text{ in } \Omega_j, \\
		  u \in W^{ 1,p }_0 \big( \Omega_j \big), \, u > 0, \text{ in } \Omega_j.
	 \end{cases}
$$

\vspace{.3cm}
In the present paper, we extend the results found in \cite{Alves} to the $ p(x)$-Laplacian operator. However, we would like emphasize that in a lot of estimates, we have used different arguments from that found in \cite{Alves}. The main difference is related to the fact that for equations involving the $p(x)$-Laplacian operator it is not clear that Moser's iteration method is a good tool to get the estimates for the $L^{\infty}$-norm. Here, we adapt some ideas explored in \cite{FanZhao0} and \cite{FuscoSbordone} to get these estimates. For more details see Section 5. 

Since we intend to find nonnegative solutions, throughout this paper, we replace $ f $ by $ f^+ \colon \Bbb R^N \times \Bbb R \to \Bbb R $ given by
$$
    f^+(x,t) =
		\begin{cases}
		   f(x,t), \text { if } t > 0 \\
			 \qquad \, 0, \text{ if } t \le 0.
		\end{cases}
$$
Nevertheless, for the sake of simplicity, we still write $ f $ instead of $ f^+ $.   

\vspace{0.5 cm}
The main theorem in this paper is the following:

\begin{theorem} \label{main}
   Assume that $ (H_1)-(H_4) $ and $ (f_1)-(f_5) $ hold. Then, there exist $ \lambda_0 > 0 $ with the following property: for any non-empty subset $ \Upsilon $ of $ \{1, 2, . . . , k \} $ and $ \lambda \ge \lambda_0 $,  problem $ \big( P_\lambda \big) $ has a solution $u_\lambda$. Moreover, if we fix the subset $ \Upsilon $, then for any sequence $ \lambda_n \to \infty $ we can extract a subsequence $ (\lambda_{n_i}) $ such that $ (u_{ \lambda_{n_i} }) $ converges strongly in $ W^{ 1,p(x) } \big( \mathbb R^N \big) $ to a function $ u $, which satisfies $ u = 0 $ outside $ \Omega_\Upsilon = \bigcup_{ j \in \Upsilon } \Omega_j $ and $ u_{|_{\Omega_j}}, \, j \in \Upsilon $, is a least energy solution for
$$
   \begin{cases}
		  - \Delta_{ p(x) } u + Z(x) u  = f(x,u), \text{ in } \Omega_j, \\
		  u \in W^{ 1,p(x) }_0 \big( \Omega_j \big), \, u \ge 0, \text{ in } \Omega_j.
	 \end{cases}
$$
\end{theorem}

\vspace{0.6 cm}

\noindent {\bf Notations:} The following notations will be used in the present work: \\

\noindent $\bullet$ \,\, $C$ and $C_i$ will denote generic positive constant, which may vary from line to line; \\

\noindent $\bullet$ \,\, In all the integrals we omit the symbol $ dx $. \\

\noindent $\bullet$ \,\, If $u$ is a mensurable function, we denote $u^+$ and $u^-$ its positive and negative part, i.e., $u^+(x) = \max\{ u(x), 0 \}$ and $ u^-(x) = \min\{ u(x), 0 \}. $ \\

\noindent $\bullet$ \,\, For $ u,v \in C \big( \mathbb R^N \big) $, the notation $ u  \ll  v $  means that $ \displaystyle \inf_{ x \in \mathbb R^N } \big( v(x)-u(x) \big) > 0 $, $ u_- = \displaystyle \inf_{ x \in \mathbb R^N } u(x) $. Moreover, we will denote by $u^*$ the function   
$$
  u^*(x) =
	\begin{cases}
	   \frac{Nu(x)}{N-u(x)}, \text{ if } u(x) < N, \\
		 \ \ \quad \infty, \text{ if } u(x) \ge N.
	\end{cases}
$$


\section{Preliminaries on variable exponents Lebesgue and Sobolev spaces}


In this section, we recall some results on variable exponents Lebesgue and Sobolev spaces found in \cite{AlvesFerreira1,FZ,FanShenZhao} and their references.   

Let $ h \in L^\infty \big( \mathbb R^N \big) $ with $ h_- = \text{ess} \displaystyle \inf_{ \! \! \! \! \! \mathbb R^N} h \ge 1$. The \emph{variable exponent Lebesgue space} $ L^{ h(x) } \big( \mathbb R^N \big) $ is defined by
\[
    L^{ h(x) } \big( \mathbb R^N \big) = \left\{ u \colon \mathbb R^N \to \mathbb R \, ; \, u \text{ is measurable and } \int_{ \mathbb R^N} \left\vert u \right\vert^{ h(x) } < \infty \right\},
\]
endowed with the norm
\[
    \left\vert u \right\vert _{ h(x) } = \inf \left\{ \lambda > 0 \, ; \, \int_{ \mathbb R^N } \left\vert \frac{u}{\lambda} \right\vert^{ h(x) } \le 1 \right\}.
\]
The \emph{variable exponent Sobolev space} is defined by
\[
   W^{ 1,h(x) } \big( \mathbb R^N \big) = \left\{  u \in L^{ h(x) } \big( \mathbb R^N \big) \, ; \, \big\vert \nabla u \big\vert \in L^{ h(x) } \big( \mathbb R^N \big) \right\},
\]
with the norm
\[
   \left\Vert u \right\Vert _{ 1,h(x) } = \inf \left\{ \lambda > 0 \, ; \, \int_{ \mathbb R^N } \left( \left\vert \frac{\nabla u}{\lambda} \right\vert^{ h(x) } + \left\vert \frac{u}{\lambda} \right\vert^{ h(x) } \right) \le 1 \right\}.
\]
If $ h_- > 1 $, the spaces $ L^{ h(x) } \big( \mathbb R^N \big) $ and $ W^{ 1,h(x) } \big( \mathbb R^N \big) $ are separable and reflexive with these norms.

We are mainly interested in subspaces of $ W^{ 1,h(x) } \big( \mathbb R^N \big) $ given by
$$
   E_W = \left\{ u \in W^{ 1,h(x) } \big( \mathbb R^N \big) \, ; \, \int_{ \mathbb R^N } W(x) |u|^{ h(x) } < \infty \right\},
$$
where $ W \in C \big( \mathbb R^N \big) $ such that $ W_- > 0 $. Endowing $ E_W $ with the norm
\[
   \left\Vert u \right\Vert _W = \inf \left\{ \lambda > 0 \, ; \, \int_{ \mathbb R^N } \left( \left\vert \frac{\nabla u}{\lambda} \right\vert^{ h(x) } + W(x)\left\vert \frac{u}{\lambda} \right\vert^{ h(x) } \right) \le 1 \right\},
\]
$ E_W $ is a Banach space. Moreover, it is easy to see that $ E_W \hookrightarrow W^{ 1,h(x) } \big( \mathbb R^N \big) $ continuously. In addition, we can show that $ E_W $ is reflexive. For the reader's convenience, we recall some basic results.

\begin{proposition} \label{normxmodular}
   The functional $ \varrho \colon E_W \to \mathbb R $ defined by
\begin{equation}
   \varrho(u) = \int_{ \mathbb R^N } \left(  \big\vert \nabla u \big\vert^{ h(x) } + W(x) \left\vert u \right\vert^{ h(x) } \right),
\end{equation}
has the following properties:
\begin{enumerate}
    \item[(i)] If $ \left\Vert u \right\Vert_W \ge 1 $, then $ \left\Vert u \right\Vert_W^{ h_- } \le \varrho(u) \le \left\Vert u\right\Vert_W^{ h_+ } $.
    \item[(ii)] If $ \left\Vert u \right\Vert_W \le 1$, then $ \left\Vert u \right\Vert_W^{ h_+ }\le \varrho(u) \le \left\Vert u \right\Vert_W^{h_- }$.
\end{enumerate}
In particular, for a sequence $ (u_n) $ in $ E_W $,
\begin{gather*}
   \left\Vert u_n \right\Vert_W \rightarrow 0 \iff \varrho(u_n) \rightarrow 0, \ \text{and}, \\
   (u_n) \ \text{is bounded in} \ E_W \iff \varrho(u_n) \ \text{is bounded in} \ \mathbb R.
\end{gather*}
\end{proposition}

\begin{remark}
   For the functional $ \varrho_{ h(x) } \colon L^{ h(x) } \big( \mathbb R^N \big) \rightarrow \mathbb R $ given by
\[
   \varrho_{ h(x) }(u) = \int_{ \mathbb R^N } \left\vert u \right\vert^{ h(x) },
\]
the same conclusion of Proposition \ref{normxmodular} also holds.
\end{remark}

\begin{proposition} \label{estimate0}
   Let $ m \in L^\infty \big( \mathbb R^N \big) $ with $ 0 < m_- \le m(x) \le h(x) \text{ for a.e. } x \in \mathbb R^N $. If $ u \in L^{ h(x) } \big( \mathbb R^N \big) $, then $ |u|^{ m(x) } \in L^{ \frac{h(x)}{m(x)} } \big( \mathbb R^N \big) $ and
$$
	 \left| |u|^{ m(x) } \right|_{ \frac{h(x)}{m(x)} } \le \max \left\{ |u|_{ h(x) }^{ m_- }, |u|_{ h(x) }^{ m_+ } \right\} \le |u|_{ h(x) }^{ m_- } + |u|_{ h(x) }^{ m_+ }.
$$
\end{proposition}

\vspace{.3cm}
Related to the Lebesgue space $ L^{ h(x) } \big( \mathbb R^N \big) $, we have the following generalized H\"{o}lder's inequality.

\begin{proposition} [H\"older's inequality] \label{Holder's inequality}
   If $ h_- > 1 $, let $ h' \colon \mathbb R^N \to \mathbb R $ such that
$$
   \frac{1}{h( x)} + \frac{1}{h'(x)} = 1 \text{ for a.e. } x \in \mathbb R^N.
$$
Then, for any $ u\in L^{ h( x) } \big( \mathbb R^N \big) $ and  $ v \in L^{ h'(x) } \big( \mathbb R^N \big) $,
$$
   \int_{\mathbb R^N} | uv | \, dx  \le \left( \frac{1}{h_-} + \frac{1}{h'_-} \right) |u| _{ h( x) } |v|_{ h'(x) }.
$$
\end{proposition}

\vspace{.3cm}
We can define \emph{variable exponent Lebesgue spaces with vector values}. We say $ u = ( u_1, \ldots, u_L ) \colon \mathbb R^N \to \mathbb R^L \in L^{ h(x) } \big( \mathbb R^N, \mathbb R^L \big) $ if, and only if, $ u_i \in L^{ h(x) } \big( \mathbb R^N \big) $, for $ i = 1, \ldots, L $. On $ L^{ h(x) } \big( \mathbb R^N, \mathbb R^L \big) $, we consider the norm $ \displaystyle | u |_{ L^{ h(x) }( \mathbb R^N, \mathbb R^L) } = \sum_{i=1}^L | u_i |_{ h(x) } $.

We state below lemmas of Brezis-Lieb type. The proof of the two first results follows the same arguments explored at \cite{Kavian}, while the proof of the latter can be found at \cite{AlvesFerreira1}.

\begin{proposition} [Brezis-Lieb lemma, first version] \label{first Brezis-Lieb}
   Let $ \left( u_n \right) $ be a bounded sequence in $ L^{ h(x) } \big( \mathbb R^N, \mathbb R^L \big) $  such that $ u_n(x) \to u(x) \text{ for a.e. } x \in \mathbb R^N $.
Then, $ u \in L^{ h(x) } \big( \mathbb R^N, \mathbb R^L \big) $ and
\begin{equation}
   \int_{ \mathbb R^N } \left| \left| u_n \right|^{ h(x) } - \left| u_n - u \right|^{ h(x) } - \left| u \right|^{ h(x) } \right| \, dx = o_n(1).
\end{equation}
\end{proposition}

\begin{proposition} [Brezis-Lieb lemma, second version] \label{second Brezis-Lieb}
   Let $ \left( u_n \right) $ be a bounded se\-quen\-ce in $ L^{ h(x) } \big( \mathbb R^N, \mathbb R^L \big) $ with $ h_- > 1 $ and $ u_n(x) \to u(x) \text{ for a.e. } x \in \mathbb R^N
$. Then
$$
   u_n \rightharpoonup u \text{ in } L^{ h(x) } \big( \mathbb R^N, \mathbb R^L \big).
$$
\end{proposition}

\begin{proposition} [Brezis-Lieb lemma, third version] \label{third Brezis-Lieb}
   Let  $ (u_n) $ be a bounded sequence in $ L^{ h(x) } \big( \mathbb R^N, \mathbb R^L \big) $ with $ h_- > 1 $ and $ u_n(x) \to u(x) $ for a.e. $ x \in \mathbb R^N $. Then
\begin{equation}
   \int_{ \mathbb R^N } \left| \left| u_n \right|^{ h(x)-2 } u_n - \left| u_n - u \right|^{ h(x)-2 } \left( u_n - u \right) - | u |^{ h(x)-2 } u \right|^{ h'(x) } \, dx = o_n(1),
\end{equation}
\end{proposition}

\vspace{.3cm}
To finish this section, we notice that for any open subset $ \Omega \subset \mathbb R^N $, we can define of the same way the spaces $ L^{ h(x) } \big( \Omega \big) $ and $ W^{ 1,h(x) } \big( \Omega \big) $. Moreover, all the above propositions hold for these spaces and, besides, we have the following embedding Theorem of Sobolev's type.

\begin{proposition} [{\cite[Theorems 1.1, 1.3]{FanShenZhao}}] \label{embedding theorem}
   Let $ \Omega \subset \mathbb R^N $ an open domain with the cone property, $ h \colon \overline{\Omega} \to \mathbb R $ satisfying $ 1 < h_- \le h_+ < N $ and $ m \in L^{\infty}_+ \big( \Omega \big) $.
  \begin{enumerate}
     \item[(i)] If $ h $ is Lipschitz continuous and $ h \le m \le h^{\ast} $, the embedding $ W^{ 1,h(x) } \big( \Omega \big) \hookrightarrow L^{ m(x) } \big( \Omega \big) $ is continuous;
     \item[(ii)] If $ \Omega $ is bounded, $ h $ is continuous and $ m \ll h^{\ast} $, the embedding $ W^{ 1,h(x) } \big( \Omega \big) \hookrightarrow L^{ m(x) } \big( \Omega \big) $ is compact.
  \end{enumerate}
\end{proposition}


\section{An auxiliary problem}


In this section, we work with an auxiliary problem adapting the ideas explored in del Pino \& Felmer \cite{DelPinoFelmer} (see also \cite{Alves}). 

We start noting that the energy functional $ I_\lambda \colon E_\lambda \to \mathbb R $ associated with $ \big( P_\lambda \big) $ is  given by
$$
   I_\lambda (u) = \int_{ \mathbb R^N } \frac{1}{p(x)} \left( \big| \nabla u \big|^{ p(x) } + \big( \lambda V(x) + Z(x) \big) | u |^{ p(x) } \right) - \int_{ \mathbb R^N } F(x,u),
$$
where $ E_\lambda = \big( E, \| \cdot \|_\lambda \big) $ with
$$
   E = \left\{ u \in W^{ 1,p(x) } \big( \mathbb R^N \big) \, ; \, \int_{ \mathbb R^N } V(x) |u|^{ p(x) } < \infty \right\},
$$
and
$$
   \| u \|_\lambda = \inf \left\{ \sigma > 0 \, ; \, \varrho_{ \lambda } \left( \frac{u}{\sigma} \right) \le 1 \right\}, 
$$
being 
$$
\varrho_{ \lambda }(u) = \int_{ \mathbb R^N } \left( \big| \nabla u \big|^{ p(x) } + \big( \lambda V(x) + Z(x) \big) | u |^{ p(x) } \right).
$$

Thus $ E_\lambda \hookrightarrow W^{ 1,p(x) } \big( \mathbb R^N \big) $ continuously for $ \lambda \ge 1 $ and $ E_\lambda $ is compactly embedded in $ L_{ loc }^{ h(x) } \big( \mathbb R^N \big) $, for all $ 1 \le h \ll p^* $. In addition, we can show that $ E_\lambda $ is a reflexive space. Also, being $ {\cal O} \subset \mathbb R^N $ an open set, from the relation
\begin{equation} \label{modular relation 1}
   \varrho_{ \lambda, {\cal O} }(u) = \int_{ \cal O } \left( \big| \nabla u \big|^{ p(x) } + \big( \lambda V(x) + Z(x) \big) | u |^{ p(x) } \right)  \ge M \int_{ \cal O } |u|^{ p(x) } = M \varrho_{ p(x), \cal{O} }(u),
\end{equation}
for all $ u \in E_\lambda $ with $ \lambda \ge 1 $, writing $ M = ( 1-\delta )^{ -1 } \nu $, for some $ 0 < \delta < 1 $ and $ \nu > 0 $, we derive
\begin{equation} \label{modular relation 2}
   \varrho_{ \lambda, \cal{O} }(u) - \nu \varrho_{ p(x), \cal{O} }(u) \ge \delta \varrho_{ \lambda,\cal{O} }(u), \, \forall u \in E_\lambda, \, \lambda \ge 1.
\end{equation}

\begin{remark}
   From the above commentaries,  in this work the parameter $\lambda$ will be always bigger than or equal to 1. 
\end{remark}

We recall that for any $ \epsilon > 0 $, the hypotheses $ (f_1) $, $ (f_2) $ and $ (f_5) $ yield
\begin{equation} \label{f estimate}
   f(x,t) \le \epsilon | t |^{ p(x)-1 } + C_\epsilon | t |^{ q(x)-1 }, \, \forall x \in \mathbb R^N, \, t \in \mathbb R,
\end{equation}
and, consequently,
\begin{equation} \label{F estimate}
   F(x,t) \le \epsilon | t |^{ p(x) } + C_\epsilon | t |^{ q(x) }, \, \forall x \in \mathbb R^N, \, t \in \mathbb R,
\end{equation}
where $ C_\epsilon $ depends on $ \epsilon $. Moreover, for each $\nu >0$ fixed, the assumptions $ (f_2) $ and $ (f_3) $ allow us considering the function $ a \colon \mathbb R^N \to \mathbb R$ given by    
\begin{equation} \label{fc a}
   a(x) = \min \left\{ a > 0 \, ; \, \frac{f(x,a)}{a^{ p(x)-1 }} = \nu \right\}.
\end{equation}
From $(f_2)$, it follows that 
\begin{equation} \label{a_}
0 < a_- = \inf_{ x \in \mathbb R^N } a(x).
\end{equation}

Using the function $a(x)$, we set the function $ \tilde{f} \colon \mathbb R^N \times \mathbb R \to \mathbb R $ given by		
$$
   \tilde{f}(x,t) =
   \begin{cases}
      \ \, f(x,t), \ t \le a(x) \\
	    \nu t^{ p(x)-1 }, \ t \ge a(x)
   \end{cases},
$$
which fulfills the inequality 
\begin{equation} \label{til f estimate}
   \tilde{f}(x,t) \le \nu | t |^{ p(x)-1 }, \, \forall x \in \mathbb R^N, t \in \mathbb R.
\end{equation}
Thus
\begin{equation} \label{t til f estimate}
   \tilde{f}(x,t) t \le \nu | t |^{ p(x) }, \, \forall x \in \mathbb R^N, t \in \mathbb R,
\end{equation}
and
\begin{equation} \label{til F estimate}
  \tilde{F}(x,t) \le \frac{\nu}{p(x)} | t |^{ p(x) }, \, \forall x \in \mathbb R^N, t \in \mathbb R,
\end{equation}
where $ \tilde{F}(x,t) = \int_0^t \tilde{f}(x,s) \, ds $.

Now, once that $ \Omega=\text{int } V^{ -1 } (0) $ is formed by $ k $ connected components $ \Omega_1, \ldots, \Omega_k $ with $ \text{dist} \big( \Omega_i, \Omega_j \big) > 0, \, i \ne j $, then for each $ j \in \{ 1, \ldots, k \} $, we are able to fix a smooth bounded domain $ \Omega'_j $ such that
\begin{equation} \label{omega}
   \overline{\Omega_j} \subset \Omega'_j \, \text{ and } \, \overline{\Omega'_i} \cap \overline{\Omega'_j} = \emptyset, \text{ for } i \ne j.
\end{equation}

From now on, we fix a non-empty subset $ \Upsilon \subset \left\{ 1, \ldots, k \right\} $ and 
$$
   \Omega_\Upsilon = \bigcup_{ j \in \Upsilon } \Omega_j, \, \Omega'_\Upsilon = \bigcup_{ j \in \Upsilon } \Omega'_j, \,
	 \chi_\Upsilon =
	 \begin{cases}
	    1, \text{ if } x \in \Omega'_\Upsilon \\
		  0, \text{ if } x \notin \Omega'_\Upsilon .
	 \end{cases}
$$
Using the above notations, we set the functions
$$
   g(x,t) = \chi_\Upsilon(x) f(x,t) + \big( 1-\chi_\Upsilon(x) \big) \tilde{f}(x,t), \, (x,t) \in \mathbb R^N \times \mathbb R
$$
and
$$
   G(x,t) = \int_0^t g(x,s) \, ds, \, (x,t) \in \mathbb R^N \times \mathbb R,
$$
and the auxiliary problem
$$
   \big( A_\lambda \big) \,
	 \begin{cases}
	    - \Delta_{ p(x) } u + \big( \lambda V(x) + Z(x) \big) | u |^{ p(x)-2 } u  = g(x,u), \text{ in } \mathbb R^N, \\
		 u \in W^{ 1,p(x) } \big( \mathbb R^N \big).
	 \end{cases}
$$

The problem $ \big( A_\lambda \big) $ is related to $ \big( P_\lambda \big) $, in the sense that, if $ u_\lambda $ is a solution for $ \big( A_\lambda \big) $  verifying
$$
    u_\lambda (x) \le a(x), \, \forall x \in \mathbb R^N \setminus \Omega'_\Upsilon,
$$
then it is a solution for $ \big( P_\lambda \big) $. 

In comparison to $ \big( P_\lambda \big) $, problem $ \big( A_\lambda \big) $ has the advantage that the energy functional associated with $ \big( A_\lambda \big) $, namely, $ \phi_\lambda \colon E_\lambda \to \mathbb R $ given by 
$$
    \phi_\lambda(u) = \int_{ \mathbb R^N } \frac{1}{p(x)} \left( \left| \nabla u \right|^{ p(x) } + \big( \lambda V(x) + Z(x) \big) | u |^{ p(x) } \right) - \int_{ \mathbb R^N } G(x,u),
$$
satisfies the $ (PS) $ condition, whereas $ I_\lambda $ does not necessarily satisfy this condition. This way, the mountain pass level (see Theorem \ref{TAlambda}) is a critical value for $ \phi_\lambda $.

\vspace{0.5 cm}

\begin{proposition} \label{mpg}
   $ \phi_\lambda $ satisfies the mountain pass geometry. 
\end{proposition}

\begin{proof}
   From (\ref{F estimate}) and (\ref{til F estimate}), 
$$
   \phi_\lambda(u) \ge \frac{1}{p_+} \varrho_\lambda (u) - \epsilon \int_{ \mathbb R^N } |u|^{ p(x) } - C_\epsilon \int_{ \mathbb R^N } |u|^{ q(x) } - \frac{\nu}{p_-} \int_{\mathbb R^N} |u|^{ p(x) }, 	
$$
for $ \epsilon > 0 $ and $ C_\epsilon > 0 $ be a constant depending on $ \epsilon $. By \eqref{modular relation 1}, fixing $ \epsilon < \frac{M}{p_+} $ and $ \nu < p_- M \left( \frac{1}{p_+}-\frac{\epsilon}{M} \right) $ and assuming $ \| u \|_\lambda < \min \left\{ 1, 1/C_q \right\} $, where $ | v |_{ q(x) } \le C_q \| v \|_\lambda, \, \forall v \in E_\lambda $, we derive from Proposition \ref{normxmodular}
$$
   \phi_\lambda (u) \ge \alpha \| u \|^{ p_+ }_\lambda - C \| u \|^{q_-}_\lambda,
$$
where $ \alpha = \left( \frac{1}{p_+} - \frac{\epsilon}{M} \right) - \frac{\nu}{p_-M} > 0 $. Once $ p_+ < q_- $, the first part of the mountain pass geometry is satisfied. Now, fixing $v \in C^{\infty}_{0}(\Omega_\Upsilon)$, we have for $t \geq 0$
$$
\phi_{\lambda}(tv)= \int_{ \mathbb R^N } \frac{t^{p(x)}}{p(x)} \left( \left| \nabla v \right|^{ p(x) } + Z(x) \big) | v |^{ p(x) } \right) - \int_{ \mathbb R^N } F(x,tv).
$$ 
If $t>1$, by $(f_3)$,  
$$
\phi_{\lambda}(tv) \leq \frac{t^{p^{+}}}{p_{-}} \int_{ \mathbb R^N }\left( \left| \nabla v \right|^{ p(x) } + Z(x) \big) | v |^{ p(x) } \right) -C_1t^{\theta}\int_{\mathbb{R}^{N}}|v|^{\theta}-C_2,
$$
and so, 
$$
\phi_\lambda(tv) \to -\infty \quad \mbox{as} \quad t \to +\infty.
$$
The last limit implies that $\phi_{\lambda}$ verifies the second geometry of the mountain pass.  \end{proof}

\begin{proposition} \label{boundedness}
   All $ (PS)_d $ sequences for $ \phi_\lambda $ are bounded in $ E_\lambda $.
\end{proposition}

\begin{proof}
   Let $ (u_n) $ be a $ (PS)_d $ sequence for $ \phi_\lambda $. So, there is $ n_0 \in \mathbb N $ such that
$$
   \phi_\lambda (u_n) - \frac{1}{\theta} \phi_\lambda'(u_n) u_n \le d+1 + \| u_n \|_\lambda, \text { for } n \ge n_0.
$$
On the other hand, by (\ref{t til f estimate}) and (\ref{til F estimate})
$$
   \tilde{F}(x,t) - \frac{1}{\theta} \tilde{f}(x,t)t \le \left( \frac{1}{p(x)} - \frac{1}{\theta} \right) \nu | t |^{ p(x) }, \, \forall x \in \mathbb R^N, t \in \mathbb R,
$$
which together with (\ref{modular relation 2}) gives 
$$
   \phi_\lambda (u_n) - \frac{1}{\theta} \phi_\lambda'(u_n) u_n \ge \left( \frac{1}{p_+} - \frac{1}{\theta} \right) \delta \varrho_\lambda (u_n), \, \forall n \in \mathbb N.
$$
Hence		
$$
   d+1 + \max \left\{ {\varrho_\lambda (u_n)}^{1/p_-}, {\varrho_\lambda (u_n)}^{1/p_+} \right\} \ge \left( \frac{1}{p_+} - \frac{1}{\theta} \right) \delta \varrho_\lambda (u_n), \, \forall n \ge n_0,
$$
from where it follows that $ (u_n) $ is bounded in $ E_\lambda $.
\end{proof}

\begin{proposition} \label{Estimativa no infinito}
If $(u_n)$ is a $(PS)_d$ sequence for $\phi_{\lambda}$, then given $\epsilon>0$, there is $R>0$ such that
\begin{equation} \label{Estimativa}
   \limsup_n \int_{ \mathbb R^N \setminus B_R (0) } \left( \big| \nabla u_n \big|^{ p(x) } + \big( \lambda V(x) + Z(x) \big) | u_n |^{ p(x) } \right) < \epsilon.
\end{equation} 
Hence, once that $g$ has a subcritical growth, if $ u \in E_\lambda $ is the weak limit of $ (u_n) $, then
$$
\int_{\mathbb R^N}g(x,u_n)u_n\,dx \to \int_{\mathbb R^N} g(x,u)u \, dx \, \text{ and } \, \int_{\mathbb R^N} g(x,u_n)v \, dx \to \int_{\mathbb R^N} g(x,u)v \, dx, \, \forall v \in E_\lambda.
$$
\end{proposition}

\begin{proof} 
   Let $ (u_n) $ be a $ (PS)_d $ sequence for $ \phi_\lambda $, $ R > 0 $ large such that $ \Omega'_\Upsilon \subset B_{ \frac{R}{2} }(0) $ and $ \eta_R \in C^\infty \big( \Bbb R^N \big) $ satisfying
$$
   \eta_R (x) =
	 \begin{cases}
	    0, \, x \in B_{ \frac{R}{2} }(0) \\
			1, \, x \in \Bbb R^N \setminus B_R (0)
	 \end{cases},
$$
$ 0 \le \eta_R \le 1 $ and $ \big| \nabla \eta_R \big| \le \dfrac{C}{R} $, where $ C > 0 $ does not depend on $ R $. This way,
\begin{align*}
   \mbox{}  & \int_{ \Bbb R^N } \left(  \big| \nabla u_n \big|^{ p(x) } + \big( \lambda V(x) + Z(x) \big) | u_n |^{ p(x) } \right) \eta_R \\
	   = & \phi_\lambda'(u_n) \left( u_n \eta_R \right) - \int_{ \Bbb R^N } u_n \big| \nabla u_n \big|^{ p(x)-2 } \nabla u_n \cdot \nabla \eta_R + \int_{ \Bbb R^N \setminus \Omega'_\Upsilon } \tilde{f}(x,u_n) u_n \eta_R.
\end{align*}
Denoting
$$
I =\int_{ \Bbb R^N } \left(  \big| \nabla u_n \big|^{ p(x) } + \big( \lambda V(x) + Z(x) \big) | u_n |^{ p(x) } \right) \eta_R,
$$ 
it follows from (\ref{t til f estimate}), 
$$
   I \le \phi_\lambda'(u_n) \left( u_n \eta_R \right) + \frac{C}{R} \int_{ \Bbb R^N } | u_n | \big| \nabla u_n \big|^{ p(x)-1 } + \nu \int_{ \Bbb R^N } | u_n |^{ p(x) } \eta_R.
$$
Using H\"older's inequality \ref{Holder's inequality} and Proposition \ref{estimate0}, we derive
$$
	 I \le \phi_\lambda'(u_n) \left( u_n \eta_R \right) + \frac{C}{R} | u_n |_{ p(x) } \max \left\{ \big| \nabla u_n \big|^{ p_- -1 }_{ p(x) }, \big| \nabla u_n \big|^{ p_+ -1 }_{ p(x) } \right\} + \frac{\nu}{M}I.
$$	
Since $ (u_n) $ and $ \Big( \big| \nabla u_n \big| \Big) $ are bounded in $ L^{ p(x) } \big( \Bbb R^N \big) $ and $\frac{\nu}{M}=1-\delta$, we obtain
$$
   \int_{ \Bbb R^N \setminus B_R(0) } \left(  \big| \nabla u_n \big|^{ p(x) } + \big( \lambda V(x) + Z(x) \big) | u_n |^{ p(x) } \right) \le o_n(1) + \frac{C}{R}.
$$
Therefore
$$
   \limsup_n \int_{ \Bbb R^N \setminus B_R(0) } \left(  \big| \nabla u_n \big|^{ p(x) } + \big( \lambda V(x) + Z(x) \big) | u_n |^{ p(x) } \right) \le \frac{C}{R}.
$$
So, given $ \epsilon > 0 $,  choosing a $ R > 0 $ possibly still bigger, we have that $ \dfrac{C}{R} < \epsilon $, which proves (\ref{Estimativa}). Now, we will show that 
$$
\int_{\Bbb R^N}g(x,u_n)u_n \to \int_{\Bbb R^N}g(x,u)u. 
$$
Using the fact that $g(x,u)u \in L^{1}(\Bbb R^N)$ together with (\ref{Estimativa}) and Sobolev embeddings, given  $\epsilon >0$, we can choose $R>0$ such that 
$$
\limsup_{n \to +\infty}\int_{\Bbb R^N \setminus B_R(0)}|g(x,u_n)u_n| \leq \frac{\epsilon}{4} \quad \mbox{and} \quad \int_{\Bbb R^N \setminus B_R(0)}|g(x,u)u| \leq \frac{\epsilon}{4}.
$$
On the other hand, since $g$ has a subcritical growth, we have by compact embeddings
$$
\int_{B_{R}(0)}g(x,u_n)u_n \to \int_{B_{R}(0)}g(x,u)u. 
$$
Combining the above informations, we conclude that
$$
\int_{\Bbb R^N}g(x,u_n)u_n \to \int_{\Bbb R^N}g(x,u)u.
$$ 
The same type of arguments works to prove that 
$$
\int_{\Bbb R^N}g(x,u_n)v \to \int_{\Bbb R^N}g(x,u)v \quad \forall v \in E_{\lambda}.
$$ 

\end{proof}

\begin{proposition} \label{(PS) condition}
   $ \phi_\lambda $ verifies the $ (PS) $ condition. 
\end{proposition}
		
\begin{proof}
Let $ (u_n) $ be a $ (PS)_d $ sequence for $ \phi_\lambda $ and $ u \in E_\lambda $ such that $u_n \rightharpoonup u$ in $E_{\lambda}$. Thereby, by Proposition \ref{Estimativa no infinito}
$$
   \int_{ \mathbb R^N } g(x,u_n)u_n \to \int_{ \mathbb R^N } g(x,u)u \, \text{ and } \, \int_{ \mathbb R^N } g(x,u_n)v \to \int_{ \mathbb R^N } g(x,u)v, \, \forall v \in E_\lambda.
$$
Moreover, the weak limit also give
$$
\int_{ \mathbb R^N } \big| \nabla u \big|^{ p(x)-2 } \nabla u \cdot \nabla ( u_n-u )  \to 0
$$
and
$$
\int_{ \mathbb R^N }   \big( \lambda V(x) + Z(x) \big) | u |^{ p(x)-2 }u ( u_n-u )  \to 0.
$$
Now, if  
$$
   P_n^1 (x) = \left( \big| \nabla u_n \big|^{ p(x)-2 } \nabla u_n - \big| \nabla u \big|^{ p(x)-2 } \nabla u \right) \cdot \left( \nabla u_n - \nabla u \right)
$$
and
$$
	 P_n^2 (x) = \left( | u_n|^{ p(x)-2 } u_n - | u |^{ p(x)-2 } u \right) ( u_n - u ),
$$
we derive
\begin{multline*}
   \int_{ \mathbb R^N } \! \! \Big( P_n^1(x) + \big( \lambda V(x) + Z(x) \big) P_n^2(x) \Big) = \phi_\lambda'(u_n)u_n + \int_{ \mathbb R^N } \! g(x,u_n)u_n - \phi_\lambda'(u_n)u - \int_{ \mathbb R^N } \! g(x,u_n)u \\
	 - \int_{ \mathbb R^N } \left( \big| \nabla u \big|^{ p(x)-2 } \nabla u \cdot \nabla ( u_n-u ) + \big( \lambda V(x) + Z(x) \big) | u |^{ p(x)-2 }u ( u_n-u ) \right).
\end{multline*}
Recalling that $\phi_\lambda'(u_n)u_n=o_n(1)$  and $\phi_\lambda'(u_n)u=o_n(1)$, the above limits lead to 
$$
   \int_{ \mathbb R^N } \Big( P_n^1(x) + \big( \lambda V(x) + Z(x) \big) P_n^2(x) \Big) \to 0.
$$
Now, the conclusion follows as in \cite{AlvesFerreira1}.
\end{proof}		

\begin{theorem} \label{TAlambda}
   The problem $ \big( A_\lambda \big) $ has a (nonnegative) solution, for all $ \lambda \ge 1 $.
\end{theorem}

\begin{proof}
   The proof is an immediate consequence of the Mountain Pass Theorem due to Ambrosetti \& Rabinowitz \cite{AmRb}. 
\end{proof}


\section{The $ (PS)_\infty $ condition}


A sequence $ (u_n) \subset  W^{ 1,p(x) } \big( \mathbb R^N \big) $ is called a $ (PS)_\infty $ \emph{sequence for the family} $ \left( \phi_\lambda \right)_{\lambda \ge 1} $, if there is a sequence $ ( \lambda_n ) \subset  [1, \infty) $ with $ \lambda_n \to \infty $, as $ n \to \infty $, verifying 
$$
   \phi_{ \lambda_n }(u_n) \to c \text{ and } \left\| \phi'_{ \lambda_n }(u_n) \right\| \to 0, \text{ as } n \to \infty.
$$

\begin{proposition} \label{(PS) infty condition}
   Let $ (u_n) \subset W^{ 1,p(x) } \big( \mathbb R^N \big) $ be a $ (PS)_\infty $ sequence for $ \left( \phi_\lambda \right)_{\lambda \ge 1} $. Then, up to a subsequence, there exists $ u \in W^{ 1,p(x) } \big( \mathbb R^N \big) $ such that $ u_n \rightharpoonup u $ in $ W^{ 1,p(x) } \big( \mathbb R^N \big) $. Furthermore,
\begin{enumerate}
  \item[(i)] $ \varrho_{ \lambda_n } (u_n-u) \to 0 $ and, consequently, $ u_n \to u $ in $ W^{ 1,p(x) } \big( \mathbb R^N \big) $;
	\item[(ii)] $ u = 0 $ in $ \mathbb R^N \setminus \Omega_\Upsilon $, $ u \ge 0 $ and $ u_{|_{\Omega_j}}, \, j \in \Upsilon $, is a solution for
   $$
      (P_j) \;
	    \begin{cases}
		     - \Delta_{ p(x) } u + Z(x) | u |^{ p(x)-2 } u  = f(x,u), \text{ in } \Omega_j, \\
		     u \in W^{ 1,p(x) }_0 \big( \Omega_j \big);
			\end{cases}
   $$
	 \item[(iii)] $ \displaystyle \int_{\mathbb R^N} \lambda_n V(x) | u_n |^{ p(x) } \to 0 $;
	 \item[(iv)] $ \varrho_{ \lambda_n, \Omega'_j } (u_n) \to \displaystyle \int_{ \Omega_j } \left( \big| \nabla u \big|^{ p(x) } + Z(x) | u |^{ p(x) } \right), \text{ for } j \in \Upsilon $;
   \item[(v)] $ \varrho_{ \lambda_n, \mathbb R^N \setminus \Omega_\Upsilon } (u_n) \to 0 $;
	 \item[(vi)] $ \phi_{ \lambda_n } (u_n) \to \displaystyle \int_{ \Omega_\Upsilon } \frac{1}{p(x)} \left( \big| \nabla u \big|^{ p(x) } + Z(x) | u |^{ p(x) } \right) - \int_{ \Omega_\Upsilon } F(x,u) $.
\end{enumerate}
\end{proposition}

\begin{proof}
   Using the same reasoning as in the proof of Proposition \ref{boundedness}, we obtain that $ \big( \varrho_{ \lambda_n }(u_n) \big) $ is bounded in $ \mathbb R $. Then $ \big( \| u_n \|_{ \lambda_n } \big) $ is bounded in $ \mathbb R $ and $ (u_n) $ is bounded in $ W^{ 1,p(x) } \big( \mathbb R^N \big) $. So, up to a subsequence, there exists $ u \in W^{ 1,p(x) } \big( \mathbb R^N \big) $ such that
$$
   u_n \rightharpoonup u  \text{ in } W^{ 1,p(x) } \big( \mathbb R^N \big) \, \text{ and } \, u_n(x) \to u(x) \text{ for a.e. } x \in \mathbb R^N.
$$
Now, for each $ m \in \mathbb N $, we define $ C_m = \left\{ x \in \mathbb R^N \, ; \, V(x) \ge \dfrac{1}{m} \right\} $. Without loss of generality, we can assume $ \lambda_n < 2 ( \lambda_n-1 ), \, \forall n \in \mathbb N $. Thus
	       $$
		       \int_{ C_m } | u_n |^{ p(x) } \le \frac{2m}{\lambda_n} \int_{ C_m } \big( \lambda_n V(x)+Z(x) \big) | u_n |^{ p(x) } \le \frac{2m}{\lambda_n} \varrho_{ \lambda_n }(u_n) \le \frac{C}{\lambda_n}.
	       $$
By Fatou's lemma, we derive
	       $$
			     \int_{ C_m } | u |^{ p(x) } = 0,
	       $$
	       which implies that $ u = 0 $ in $ C_m $ and, consequently, $ u = 0 $ in $ \mathbb R^N \setminus \overline{\Omega} $. From this, we are able to prove $(i)-(vi)$. 
	       
\begin{enumerate}
   \item[$(i)$] Since $ u = 0 $  in $ \mathbb R^N \setminus \overline{\Omega} $, repeating the argument explored in  Proposition \ref{(PS) condition} we get
	       $$
            \int_{ \mathbb R^N } \Big( P_n^1(x) + \big( \lambda_n V(x) + Z(x) \big) P_n^2(x) \Big) \to 0,
         $$
         where
         $$
            P_n^1 (x) = \left( \big| \nabla u_n \big|^{ p(x)-2 } \nabla u_n - \big| \nabla u \big|^{ p(x)-2 } \nabla u \right) \cdot \left( \nabla u_n - \nabla u \right) \\
	        $$
					and
					$$
	          P_n^2 (x) = \left( | u_n|^{ p(x)-2 } u_n - | u |^{ p(x)-2 } u \right) ( u_n - u ).
         $$
         Therefore, $ \varrho_{ \lambda_n } ( u_n-u ) \to 0 $, which implies $ u_n \to u $ in $ W^{ 1,p(x) } \big( \mathbb R^N \big) $.
	 \item[$(ii)$] Since $ u \in W^{ 1,p(x) } \big( \mathbb R^N \big) $ and $ u = 0 $  in $ \mathbb R^N \setminus \overline{\Omega} $, we have $ u \in W^{ 1,p(x) }_0 \big( \Omega \big) $ or, equivalently, $ u_{ |_{\Omega_j} } \in W^{ 1,p(x) }_0 \big( \Omega_j \big) $, for $ j = 1, \ldots, k $. Moreover, the limit $u_n \to u$ in $W^{1,p(x)}(\mathbb R^N)$ combined with $\phi'_{\lambda_n}(u_n)\varphi \to 0$  for $\varphi \in C^{\infty}_0 \big( \Omega_j \big)$ implies that  
         \begin{equation} \label{u is solution}	
	          \int_{\Omega_j} \left( \big| \nabla u \big|^{ p(x)-2 } \nabla u \cdot \nabla \varphi + Z(x) | u |^{ p(x)-2 } u \varphi \right) - \int_{\Omega_j} g(x,u) \varphi = 0,
         \end{equation}
showing that $ u_{ |_{\Omega_j} }  $ is a solution for
         $$
				 \begin{cases}
		        - \Delta_{ p(x) } u + Z(x) | u |^{ p(x)-2 } u  = g(x,u), \text{ in } \Omega_j, \\
		        u \in W^{ 1,p(x) }_0 \big( \Omega_j \big).
		       \end{cases}
         $$
				 This way, if $ j \in \Upsilon $, then $ u_{ |_{\Omega_j} } $ satisfies $ (P_j) $. On the other hand, if $ j \notin \Upsilon $, we must have 
				 $$
				   \int_{ \Omega_j } \left( \big| \nabla u \big|^{ p(x) } + Z(x) | u |^{ p(x) } \right) - \int_{ \Omega_j } \tilde{f}(x,u)u = 0.
				 $$
				The above equality combined with (\ref{t til f estimate}) and (\ref{modular relation 2}) gives  
			 	 $$
				    0 \ge \varrho_{ \lambda, \Omega_j }(u) - \nu \varrho_{ p(x), \Omega_j }(u) \ge \delta \varrho_{ \lambda, \Omega_j }(u) \ge 0, 
				 $$
				 from where it follows $ u_{|_{\Omega_j}} = 0 $.  This proves $ u = 0 $ outside $ \Omega_\Upsilon $ and $ u \ge 0 $ in $ \Bbb R^N $.
				
	\item[$(iii)$]  It follows from (i), since
	       $$
		        \int_{ \mathbb R^N } \lambda_n V(x) | u_n |^{ p(x) } = \int_{ \mathbb R^N } \lambda_n V(x) | u_n-u |^{ p(x) } \le 2 \varrho_{ \lambda_n }(u_n-u).
	       $$
	\item[$(iv)$] Let $ j \in \Upsilon $. From (i), 
	       $$
				    \varrho_{ p(x), \Omega'_j }( u_n-u ), \varrho_{ p(x), \Omega'_j } \big( \nabla u_n - \nabla u \big) \to 0.
				 $$
				 Then by  Proposition \ref{first Brezis-Lieb}, 
				 $$
				    \int_{ \Omega'_j } \big( \big| \nabla u_n \big|^{ p(x) } - \big| \nabla u \big|^{ p(x) } \big) \to 0 \quad \mbox{and} \quad \int_{ \Omega'_j } Z(x) \big( | u_n |^{ p(x) } - | u |^{ p(x) } \big) \to 0.
				 $$
				 From (iii),
				 $$
				    \int_{ \Omega'_j } \lambda_n V(x) \big( | u_n |^{ p(x) } - | u |^{ p(x) } \big) = \int_{ \Omega'_j \setminus \overline{\Omega_j} } \lambda_n V(x) | u_n |^{ p(x) } \to 0.
			   $$
				 This way
				 $$
				   \varrho_{ \lambda_n, \Omega'_j } (u_n) - \varrho_{ \lambda_n, \Omega'_j } (u) \to 0.
				 $$
				 Once $ u = 0 \text{ in } \Omega'_j \setminus \Omega_j $, we get
				 $$
				    \varrho_{ \lambda_n, \Omega'_j } (u_n) \to \int_{ \Omega_j } \left( | \nabla u |^{ p(x) } + Z(x) | u |^{ p(x) } \right).
				 $$
	 \item[$(v)$] By (i),  $ \varrho_{ \lambda_n }( u_n-u ) \to 0 $, and so, 
	              $$
								   \varrho_{ \lambda_n, \mathbb R^N \setminus \Omega_\Upsilon }(u_n) \to 0. 
								$$
	 \item[$(vi)$] We can write the functional  $\phi_{\lambda_n}$ in the following way
	       \begin{multline*}
	          \phi_{ \lambda_n } (u_n) = \sum_{ j \in \Upsilon } \int_{ \Omega'_j } \frac{1}{p(x)} \left( \big| \nabla u_n \big|^{ p(x) } + \big( \lambda_n V(x) + Z(x) \big) | u_n |^{ p(x) } \right) \\
						+ \int_{ \mathbb R^N \setminus \Omega'_\Upsilon } \frac{1}{p(x)} \left( \big| \nabla u_n \big|^{ p(x) } + \big( \lambda_n V(x) + Z(x) \big) | u_n |^{ p(x) } \right) - \int_{ \mathbb R^N } G(x,u_n).
				 \end{multline*}
				 From $(i)-(v)$, 
				 $$
					  \int_{ \Omega'_j } \frac{1}{p(x)} \left( \big| \nabla u_n \big|^{ p(x) } + \big( \lambda_n V(x) + Z(x) \big) | u_n |^{ p(x) } \right) \to \int_{ \Omega_j } \frac{1}{p(x)} \left( \big| \nabla u \big|^{ p(x) } +  Z(x) | u |^{ p(x) } \right),
				 $$
				 $$
					 \int_{ \mathbb R^N \setminus \Omega'_\Upsilon } \frac{1}{p(x)} \left( \big| \nabla u_n \big|^{ p(x) } + \big( \lambda_n V(x) + Z(x) \big) | u_n |^{ p(x) } \right) \to 0.
				 $$
				 and 
				 $$
					  \int_{ \mathbb R^N } G(x,u_n) \to \int_{ \Omega_\Upsilon } F(x,u).
				 $$
				 Therefore
				 $$
					  \phi_{ \lambda_n } (u_n) \to \int_{ \Omega_\Upsilon } \frac{1}{p(x)} \left( | \nabla u |^{ p(x) } + Z(x) | u |^{ p(x) } \right) - \int_{ \Omega_\Upsilon } F(x,u).
				 $$
\end{enumerate}	
\end{proof}


\section{The boundedness of the $ \big( A_\lambda \big) $ solutions}


In this section, we study the boundedness outside $ \Omega'_\Upsilon $ for some solutions of $ \big( A_\lambda \big ) $. To this end, we adapt for our problem arguments found in \cite{FanZhao0} and \cite{FuscoSbordone}. 

\begin{proposition} \label{P:boundedness of the solutions}
    Let $ \big( u_\lambda \big) $ be a family of solutions for $ \big( A_\lambda \big) $ such that $ u_\lambda \to 0 $ in $ W^{ 1,p(x) } \big( \mathbb R^N \setminus \Omega_\Upsilon \big) $, as $ \lambda \to \infty $. Then, there exists $ \lambda^* > 0 $ with the following property:
$$
   \left| u_\lambda \right|_{ \infty, \mathbb R^N \setminus \Omega'_\Upsilon } \le a_-, \, \forall \lambda \ge \lambda^*.
$$
Hence, $u_{\lambda}$ is a solution for $(P_\lambda)$ for $\lambda \geq \lambda^*$.
\end{proposition}

\vspace{.3cm}
Before to prove the above proposition, we need to show some technical lemmas. 

\begin{lemma} \label{boundary's cover}
   There exist $ x_1, \ldots, x_l \in \partial \Omega'_\Upsilon $ and corresponding  $ \delta_{x_1}, \ldots, \delta_{x_l} > 0 $ such that
$$
   \partial \Omega'_\Upsilon \subset {\cal N} \left( \partial \Omega'_\Upsilon \right)  : = \bigcup_{ i=1 }^l B_{ \frac{\delta_{x_i}}{2} } (x_i).
$$
Moreover,
\begin{equation} \label{qplus x pminus}
   q^{x_i}_+ \le \big( p^{x_i}_- \big)^*,
\end{equation}
where
$$
   q^{x_i}_+ = \sup_{ B_{ \delta_{x_i} }({x_i}) } q, \ p^{x_i}_- = \inf_{ B_{\delta_{x_i}}(x_i) } p  \text{ \, and \, } \big( p^{x_i}_- \big)^* = \frac{N p^{x_i}_-}{N-p^{x_i}_-}.
$$
\end{lemma}

\begin{proof}
   From \eqref{omega}, $ \overline{\Omega_\Upsilon} \subset \Omega'_\Upsilon $. So, there is $ \delta > 0 $ such that
$$
   \overline{B_{\delta}(x)} \subset \mathbb R^N \setminus \overline{\Omega_\Upsilon}, \, \forall x \in \partial \Omega'_\Upsilon.
$$
Once $ q \ll p^* $, there exists $ \epsilon > 0 $ such that $  \epsilon \le p^*(y) - q(y) $, for all $ y \in \mathbb R^N $. Then, by continuity, for each $ x \in \partial \Omega'_\Upsilon $ we can choose a sufficiently small $ 0 < \delta_x \le \delta $ such that
$$
   q^x_+ \le \big( p^x_- \big)^*,
$$
where
$$
   q^x_+ = \sup_{ B_{ \delta_x }(x) } q, \ p^x_- = \inf_{ B_{ \delta_x }(x) } p  \text{ \, and \, } \big( p^x_- \big)^* = \frac{N p^x_-}{N-p^x_-}.
$$
Covering $ \partial \Omega'_\Upsilon $ by the balls $ B_{ \frac{\delta_x}{2} }(x), \, x \in \partial \Omega'_\Upsilon $, and using its compactness, there are $ x_1, \ldots, x_l \in \partial \Omega'_\Upsilon $ such that
$$
   \partial \Omega'_\Upsilon \subset \bigcup_{ i=1 }^l  B_{ \frac{\delta_{x_i}}{2} }(x_i).
$$
\end{proof}

\begin{lemma} \label{good estimate}
   If $ u_\lambda $ is a solution for $ \big( A_\lambda \big) $, in each $ B_{ \delta_{x_i} }(x_i), \, i = 1, \ldots, l $, given by Lemma \ref{boundary's cover}, it is fulfilled
$$
    \int_{ A_{k,\overline{\delta},x_i} } \big| \nabla u_\lambda \big|^{ p^{x_i}_- } \le C \Bigg( \big( k^{ q_+ } + 2 \big) \big| { A_{k,\widetilde{\delta},x_i} } \big| + \left(\widetilde{\delta}-\overline{\delta} \right)^{ -\big( p^{x_i}_- \big)^* } \int_{ A_{k,\widetilde{\delta},x_i} } \left( u_\lambda-k \right)^{ \big( p^{x_i}_- \big)^* } \Bigg),
$$
where $ 0 < \overline{\delta} < \widetilde{\delta} < \delta_{ x_i } $, $ k \ge \dfrac{a_-}{4} $, $ C = C \big( p_-, p_+, q_-, q_+, \nu, \delta_{ x_i } \big) > 0 $ is a constant independent of $ k $, and for any $ R > 0 $, we denote by $A_ { k,R,x_i }$ the set 
$$
    A_ { k,R,x_i } = B_R(x_i) \cap \left\{ x \in \mathbb R^N \, ; \, u_\lambda(x) > k \right\}.
$$
\end{lemma}

\begin{proof}
   We choose arbitrarily $ 0 < \overline{\delta} < \widetilde{\delta} <  \delta_{x_i} $ and $ \xi \in C^{ \infty } \big( \mathbb R^N \big) $ with
$$
   0 \le \xi \le 1,  \, \text{ supp } \xi \subset B_{ \widetilde{\delta} }(x_i), \, \xi = 1 \text{ in } B_{ \overline{\delta} }(x_i) \, \text{ and } \, \big| \nabla \xi \big| \le \frac{2}{\widetilde{\delta}-\overline{\delta}}.
$$
For $ k \ge \dfrac{a_-}{4} $, we define $ \eta = \xi^{ p_+ } ( u_\lambda-k )^+ $. We notice that
$$
   \nabla \eta = p_+ \xi^{ p_+-1 } (u_\lambda-k) \nabla \xi + \xi^{ p_+ } \nabla u_\lambda
$$
on the set $ \left\{ u_\lambda > k \right\} $. Then, writing $ u_\lambda = u $ and taking $ \eta $ as a test function, we obtain
\begin{multline*}
   p_+ \int_{ A_{k,\widetilde{\delta},x_i} } \xi^{ p_+-1 } (u-k) \big| \nabla u \big|^{ p(x)-2 } \nabla u \cdot \nabla \xi + \int_{ A_{k,\widetilde{\delta},x_i} } \xi^{ p_+ } \big| \nabla u \big|^{ p(x) } \\
	 + \int_{ A_{k,\widetilde{\delta},x_i} } \big( \lambda V(x) + Z(x) \big) u^{ p(x)-1 } \xi^{ p_+ } (u-k) =  \int_{ A_{k,\widetilde{\delta},x_i} } g(x,u) \xi^{ p_+ } (u-k).
\end{multline*}
If we set
$$
   J = \int_{ A_{k,\widetilde{\delta},x_i} } \xi^{ p_+ } \big| \nabla u \big|^{ p(x) },
$$
using that $ \nu \le \lambda V(x) + Z(x), \, \forall x \in \mathbb R^N  $, we get
\begin{multline} \label{J estimate}
   J \le p_+ \int_{ A_{k,\widetilde{\delta},x_i} } \xi^{ p_+-1 } (u-k) \big| \nabla u \big|^{ p(x)-1 } \big| \nabla \xi \big| \\
	- \int_{ A_{k,\widetilde{\delta},x_i} } \nu u^{ p(x)-1 } \xi^{ p_+ } (u-k)  + \int_{ A_{k,\widetilde{\delta},x_i} } g(x,u) \xi^{ p_+ } (u-k).
\end{multline}
From \eqref{J estimate}, \eqref{f estimate} and \eqref{til f estimate}, 
\begin{multline*}
   J \le p_+ \int_{ A_{k,\widetilde{\delta},x_i} } \xi^{ p_+-1 } (u-k) \big| \nabla u \big|^{ p(x)-1 } \big| \nabla \xi \big| - \int_{ A_{k,\widetilde{\delta},x_i} } \nu u^{ p(x)-1 } \xi^{ p_+ } (u-k) \\
	 + \int_{ A_{k,\widetilde{\delta},x_i} } \big( \nu u^{ p(x)-1 } + C_\nu u^{ q(x)-1 } \big) \xi^{ p_+ } (u-k),
\end{multline*}
from where it follows
$$
   J \le p_+ \int_{ A_{k,\widetilde{\delta},x_i} } \xi^{ p_+-1 } (u-k) \big| \nabla u \big|^{ p(x)-1 } \big| \nabla \xi \big| + C_\nu \int_{ A_{k,\widetilde{\delta},x_i} } u^{ q(x)-1 } (u-k).
$$
Using Young's inequality, we obtain, for $ \chi \in (0,1) $,
\begin{multline*}
   J \le \frac{p_+  (p_+-1)}{p_-} \chi^{ \frac{p_-}{p_+-1}} J + \frac{2^{ p_+ } p_+}{p_-} \chi^{ -p_+ } \int_{ A_{k,\widetilde{\delta},x_i} } \left( \frac{u-k}{\widetilde{\delta}-\overline{\delta}} \right)^{ p(x) } \\
	+ \frac{C_\nu (q_+-1)}{q_-} \int_{ A_{k,\widetilde{\delta},x_i} } u^{ q(x) } + \frac{C_\nu \left( 1 + \delta_{ x_i }^{ q_+ } \right)}{q_-} \int_{ A_{k,\widetilde{\delta},x_i} } \left( \frac{u-k}{\widetilde{\delta}-\overline{\delta}}\right)^{ q(x) }.
\end{multline*}
Writing
$$
   Q = \int_{ A_{k,\widetilde{\delta},x_i} } \left( \frac{u-k}{\widetilde{\delta}-\overline{\delta}} \right)^{ \left( p^{x_i}_- \right)^* },
$$
for $ \chi \approx 0^+ $ fixed, due to \eqref{qplus x pminus}, we deduce
\begin{multline*}
   J \le \frac{1}{2} J + \frac{2^{ p_+ } p_+}{p_-} \chi^{ -p_+ } \Big( \big| { A_{k,\widetilde{\delta},x_i} } \big| + Q \Big) + \frac{C_\nu 2^{ q_+ } (q_+-1) \left( 1 + \delta_{ x_i }^{ q_+ } \right)}{q_-} \Big( \big| { A_{k,\widetilde{\delta},x_i} } \big| + Q \Big) \\
	+ \frac{C_\nu 2^{ q_+ } (q_+-1) \left( 1+k^{ q_+ } \right)}{q_-} \big| { A_{k,\widetilde{\delta},x_i} } \big| + \frac{C_\nu \left( 1 + \delta_{ x_i }^{ q_+ } \right)}{q_-} \Big( \big| { A_{k,\widetilde{\delta},x_i} } \big| + Q \Big).
\end{multline*}
Therefore
$$
   \int_{ A_{k,\overline{\delta},x_i} } \big| \nabla u \big|^{ p(x) } \le J \le C \left[ \big( k^{ q_+ } + 1 \big) \big| A_{k,\widetilde{\delta},x_i} \big| + Q \right],
$$
for a positive constant $ C = C \big( p_-, p_+, q_-, q_+, \nu, \delta_{ x_i } \big) $ which does not depend on $ k $. Since
$$
   \big| \nabla u \big|^{ p^{x_i}_- } - 1 \le \big| \nabla u \big|^{ p(x) }, \, \forall x \in B_{\delta_{x_i}}(x_i),
$$
we obtain
\begin{align*}
    \int_{ A_{k,\overline{\delta},x_i} } \big| \nabla u \big|^{ p^{x_i}_- } & \le C \left[ \big( k^{ q_+ } + 1 \big) \big| A_{k,\widetilde{\delta},x_i} \big| + Q \right] + \big| A_{k,\widetilde{\delta},x_i} \big| \\
		& \le C \left( \big( k^{ q_+ } + 2 \big) \big| A_{k,\widetilde{\delta},x_i} \big| + \left(\widetilde{\delta}-\overline{\delta} \right)^{ -\big( p^{x_i}_- \big)^* } \int_{ A_{k,\widetilde{\delta},x_i} } \left( u-k \right)^{ \big( p^{x_i}_- \big)^* } \right),
\end{align*}
for a positive constant $ C = C \big( p_-, p_+, q_-, q_+, \nu, \delta_{ x_i } \big) $ which does not depend on $ k $.
\end{proof}

\vspace{.3cm}
The next lemma can be found at (\cite[Lemma 4.7]{LadyUral}).

\begin{lemma} \label{Ladyzhenskaya}
   Let $ (J_n) $ be a sequence of nonnegative numbers satisfying
$$
   J_{ n+1 } \le C B^n J_n^{ 1+\eta }, \, n=0,1,2,\ldots,
$$
where $ C, \eta > 0 $ and $ B > 1 $. If
$$
   J_0 \le C^{ - \frac{1}{\eta} } B^{ - \frac{1}{{\eta}^2} },
$$
then $ J_n \to 0 $, as $ n \to \infty $.
\end{lemma}

\begin{lemma} \label{boundedness of the solutions}
    Let $ \big( u_\lambda \big) $ be a family of solutions for $ \big( A_\lambda \big) $ such that $ u_\lambda \to 0 $ in $ W^{ 1,p(x) } \big( \mathbb R^N \setminus \Omega_\Upsilon \big) $, as $ \lambda \to \infty $. Then, there exists $ \lambda^* > 0 $ with the following property:
$$
   \left| u_\lambda \right|_{ \infty, {\cal N} \left( \partial \Omega'_\Upsilon \right)  } \le a_-, \, \forall \lambda \ge \lambda^*.
$$
\end{lemma}

\begin{proof}
   It is enough to prove the inequality in each ball $ B_{\frac{\delta_{x_i}}{2}} (x_i), \, i = 1, \ldots, l $, given by Lemma \ref{boundary's cover}. We set
$$
   \widetilde{\delta}_n = \frac{\delta_{x_i}}{2} + \frac{\delta_{x_i}}{2^{n+1}}, \ \overline{\delta}_n = \frac{\widetilde{\delta}_n + \widetilde{\delta}_{n+1}}{2}, \ k_n = \frac{a_-}{2} \left( 1-\frac{1}{2^{n+1}} \right), \, \forall n = 0, 1, 2, \ldots .
$$
Then
$$
   \widetilde{\delta}_n \downarrow \frac{\delta_{x_i}}{2}, \quad \widetilde{\delta}_{ n+1 } < \overline{\delta}_n < \widetilde{\delta}_n, \quad k_n \uparrow \frac{a_-}{2}.
$$
From now on, we fix 
$$
   J_n(\lambda) = \int_{ A_{k_n,\widetilde{\delta}_n,x_i} } \big( u_\lambda(x) - k_n \big)^{ \left( p^{x_i}_- \right)^* }, \, n = 0, 1, 2, \ldots.
$$
and $ \xi \in C^1 \big( \mathbb R \big) $ such that
$$
   0 \le \xi \le 1, \ \xi(t) = 1, \text{ for } t \le \frac{1}{2}, \text{ and } \xi(t) = 0, \text{ for } t \ge \frac{3}{4}.
$$
Setting
$$
   \xi_n(x) = \xi \Bigg( \frac{2^{ n+1 }}{\delta_{x_i}} \bigg( \big| x-x_i \big|-\frac{\delta_{x_i}}{2} \bigg) \Bigg), \, x \in \mathbb R^N, \, n = 0, 1, 2, \ldots,
$$
we have $ \xi_n = 1 $ in $ B_{ \widetilde{\delta}_{n+1} }(x_i) $ and $ \xi_n = 0 $ outside $ B_{ \overline{\delta}_n }(x_i) $. Writing $ u_\lambda = u $, we get
\begin{align*}
   J_{ n+1 } & \le \int_{ A_{k_{n+1},\overline{\delta}_n,x_i} } \big( (u(x) - k_{ n+1 } ) \xi_n(x) \big)^{ \left( p^{x_i}_- \right)^* } \\
   & = \int_{ B_{ \delta_{ x_i } }(x_i) } \big( ( u-k_{ n+1 } )^+(x) \xi_n(x) \big)^{ \left( p^{x_i}_- \right)^* }  \\
	 & \le C \big( N, p^{x_i}_- \big) \left( \int_{ B_{ \delta_{ x_i } }(x_i) } \big| \nabla \big( ( u-k_{ n+1 } )^+ \xi_n \big)(x) \big|^{ p^{x_i}_- } \right)^{ \frac{\left( p^{x_i}_- \right)^*}{ p^{x_i}_-} } \\
   & \le C \big( N, p^{x_i}_- \big)  \left( \int_{ A_{k_{n+1},\overline{\delta}_n,x_i} } \big| \nabla u \big|^{ p^{x_i}_- } + \int_{ A_{k_{n+1},\overline{\delta}_n,x_i} } ( u-k_{ n+1 } )^{ p^{x_i}_- } \big| \nabla \xi_n \big|^{ p^{x_i}_- } \right)^{ \frac{\left( p^{x_i}_- \right)^*}{ p^{x_i}_-} }.
\end{align*}
Since
$$
   \big| \nabla \xi_n(x) \big| \le C \big( \delta_{ x_i } \big) 2^{ n+1 }, \, \forall x \in \mathbb R^N,
$$
writing $ J_{ n+1 }^{ \frac{p^{x_i}_-}{\left( p^{x_i}_- \right)^*} } = \widetilde{J}_{ n+1 } $, we obtain
$$
	\widetilde{J}_{ n+1 } \le C \Big( N, p^{x_i}_-, \delta_{ x_i } \Big) \left( \int_{ A_{k_{n+1},\overline{\delta}_n,x_i} } \big| \nabla u \big|^{ p^{x_i}_- } + 2^{ n p^{x_i}_- } \int_{ A_{k_{n+1},\overline{\delta}_n,x_i} } ( u-k_{ n+1 } )^{ p^{x_i}_- } \right).
$$
Using Lemma \ref{good estimate}, 
\begin{align*}
   \widetilde{J}_{ n+1 } & \le C \Big( N, p^{x_i}_-, \delta_{ x_i } \Big) \bigg( \left( k_{ n+1 }^{ q_+ }+2 \right) \big| A_{ k_{ n+1 },\widetilde{\delta}_n,x_i } \big| \\
	 & \phantom{ \le } + \left( \frac{2^{ n+3 }}{\delta_{ x_i }} \right)^{ \left( p^{x_i}_- \right)^* } \int_{ A_{ k_{ n+1 },\widetilde{\delta}_n,x_i } } ( u-k_{ n+1 } )^{ \left( p^{x_i}_- \right)^* } + 2^{ n p^{x_i}_- } \int_{ A_{ k_{ n+1 },\widetilde{\delta}_n,x_i } } ( u-k_{ n+1 } )^{ p^{x_i}_- } \bigg) \\
	 & \le C \Big( N, p^{x_i}_-, \delta_{ x_i } \Big) \bigg( \left( k_{ n+1 }^{ q_+ }+2 \right) \big| A_{ k_{ n+1 },\widetilde{\delta}_n,x_i } \big| \\
	 & \phantom{ \le } +  2^{ n \left( p^{x_i}_- \right)^* } \int_{ A_{ k_{ n+1 },\widetilde{\delta}_n,x_i } } ( u-k_{ n+1 } )^{ \left( p^{x_i}_- \right)^* } + 2^{ n p^{x_i}_- } \int_{ A_{ k_{ n+1 },\widetilde{\delta}_n,x_i } } ( u-k_{ n+1 } )^{ p^{x_i}_- } \bigg).
\end{align*}
From Young's inequality 
$$
   \int_{ A_{ k_{ n+1 },\widetilde{\delta}_n,x_i } } ( u-k_{ n+1 } )^{ p^{x_i}_- } \le C \Big( p^{x_i}_- \Big) \left( \big| A_{ k_{ n+1 },\widetilde{\delta}_n,x_i } \big| + \int_{ A_{ k_{ n+1 },\widetilde{\delta}_n,x_i } } ( u-k_{ n+1 } )^{ \left( p^{x_i}_- \right)^* } \right).
$$
Thus
$$
   \widetilde{J}_{ n+1 } \le C \Big( N, p^{x_i}_-, \delta_{ x_i } \Big) \Bigg( \bigg( \left( \frac{a_-}{2} \right)^{ q_+ }+2+2^{ n p^{x_i}_- } \bigg) \big| A_{ k_{ n+1 },\widetilde{\delta}_n,x_i } \big| + 2^{ n \left( p^{x_i}_- \right)^* } J_n + 2^{ n p^{ x_i }_- } J_n \Bigg).
$$
Now, since
$$
   J_n \ge \int_{ A_{ k_{ n+1 },\widetilde{\delta}_n,x_i } } ( u-k_n)^{ \left( p^{x_i}_- \right)^* } \ge ( k_{ n+1 }-k_n )^{ \left( p^{x_i}_- \right)^* } \big| A_{ k_{ n+1 },\widetilde{\delta}_n,x_i } \big|
$$
it follows that
$$
   \big| A_{ k_{ n+1 },\widetilde{\delta}_n,x_i } \big| \le \left( \frac{2^{ n+3 }}{a_-} \right)^{ \left( p^{x_i}_- \right)^* } J_n,
$$
and so,
\begin{align*}
   \widetilde{J}_{ n+1 } & \le C \Big( N, p^{x_i}_-, \delta_{ x_i }, a_-, q_+ \Big) \left( 2^{ n \left( p^{x_i}_- \right)^* } J_n + 2^{ n \big( p^{ x_i }_- + \left(  p^{x_i}_- \right)^*  \big) } J_n + 2^{ n \left( p^{x_i}_- \right)^* } J_n + 2^{ n p^{ x_i }_- } J_n \right).
\end{align*}
Fixing $ \alpha = \big( p^{ x_i }_- + \left(  p^{x_i}_- \right)^*  \big) $, it follows that 
$$
   J_{ n+1 } \le C \Big( N, p^{x_i}_-, \delta_{ x_i }, a_-, q_+ \Big) \left( 2^{ \alpha \frac{\left( p^{x_i}_- \right)^*}{p^{ x_i }_-} } \right)^n { J_n }^{ \frac{\left( p^{x_i}_- \right)^*}{p^{ x_i }_-} },
$$
and consequently
$$
   J_{ n+1 } \le C B^n J_n^{ 1+\eta },
$$
where $ C = C \Big( N, p^{x_i}_-, \delta_{ x_i }, a_-, q_+ \Big) $, $ B = 2^{ \alpha \frac{\left( p^{x_i}_- \right)^*}{p^{ x_i }_-} } $ and $ \eta = \frac{\left( p^{x_i}_- \right)^*}{p^{ x_i }_-} -1 $. Now, once that $ u_\lambda \to 0 $ in $ W^{ 1,p(x) } \big( \mathbb R^N \setminus \Omega_\Upsilon \big) $, as $ \lambda \to \infty $,  there exists $ \lambda_i > 0 $ such that
$$
   \int_{ A_{ \frac{a_-}{4}, \delta_{ x_i }, x_i } } \left( u_\lambda-\frac{a_-}{4} \right)^{ \left( p^{x_i}_- \right)^* } = J_0(\lambda) \le C^{ - \frac{1}{\eta} } B^{ - \frac{1}{{\eta}^2} }, \quad \lambda \geq \lambda_i. 
$$
From Lemma \ref{Ladyzhenskaya}, $ J_n(\lambda) \to 0 $, $ n \to \infty $, for all $ \lambda \geq \lambda_i$, and so, 
$$
    u_\lambda \le \frac{a_-}{2} < a_-, \text{ in } B_{\frac{\delta_{x_i}}{2}}, \text{ for all }  \lambda \ge \lambda_i.
$$
Now, taking $ \lambda^* = \max \{ \lambda_1, \ldots, \lambda_l \} $, we conclude that
$$
    \left| u_\lambda \right|_{ \infty, {\cal N} \left( \partial \Omega'_\Upsilon \right) } < a_-, \, \forall \lambda \ge \lambda^*.
$$

\end{proof}

\vspace{.3cm}
\begin{proof} [Proof of Proposition \ref{P:boundedness of the solutions}]
  Fix $ \lambda \ge \lambda^* $, where $ \lambda^* $ is given at Lemma \ref{boundedness of the solutions}, and define $ \widetilde{u}_\lambda \colon \mathbb R^N \setminus \Omega'_\Upsilon \to \mathbb R $ given by
$$
   \widetilde{u}_\lambda(x) = \left( u_\lambda-a_- \right)^+ (x).
$$
From Lemma \ref{boundedness of the solutions},  $ \widetilde{u}_\lambda \in W^{ 1,p(x) }_0 \big( \mathbb R^N \setminus \Omega'_\Upsilon \big) $. Our goal is showing that  $\widetilde{u}_\lambda = 0 $ in $ \mathbb R^N \setminus \Omega'_\Upsilon $. This implies
$$
   \left| u_\lambda \right|_{ \infty, \mathbb R^N \setminus \Omega'_\Upsilon } \le a_-.
$$
In fact, extending $ \widetilde{u}_\lambda = 0 $ in $ \Omega'_\Upsilon $ and taking $ \widetilde{u}_\lambda $ as a test function, we obtain
$$
   \int_{ \mathbb R^N \setminus \Omega'_\Upsilon } \! \! \big| \nabla u_\lambda \big|^{ p(x)-2 } \nabla u_\lambda \cdot \nabla \widetilde{u}_\lambda + \int_{ \mathbb R^N \setminus \Omega'_\Upsilon } \! \! \! \! \big( \lambda V(x) + Z(x) \big) u_{\lambda}^{ p(x)-2 } u_\lambda \widetilde{u}_\lambda = \! \! \int_{ \mathbb R^N \setminus \Omega'_\Upsilon } g \left( x, u_\lambda \right) \widetilde{u}_\lambda.
$$
Since
\begin{gather*}
   \int_{ \mathbb R^N \setminus \Omega'_\Upsilon } \big| \nabla u_\lambda \big|^{ p(x)-2 } \nabla u_\lambda \cdot \nabla \widetilde{u}_\lambda = \int_{ \mathbb R^N \setminus \Omega'_\Upsilon } \big| \nabla \widetilde{u}_\lambda \big|^{ p(x) }, \\
	\int_{ \mathbb R^N \setminus \Omega'_\Upsilon } \! \! \! \!\big( \lambda V(x) + Z(x) \big) u_{\lambda}^{ p(x)-2 } u_\lambda \widetilde{u}_\lambda = \int_{ \left( \mathbb R^N \setminus \Omega'_\Upsilon \right)_+ } \! \! \! \!	\big( \lambda V(x) + Z(x) \big) u_{\lambda}^{ p(x)-2 } \left( \widetilde{u}_\lambda+a_- \right) \widetilde{u}_\lambda
\end{gather*}
and
$$
  \int_{ \mathbb R^N \setminus \Omega'_\Upsilon } g \left( x, u_\lambda \right) \widetilde{u}_\lambda = \int_{ \left( \mathbb R^N \setminus \Omega'_\Upsilon \right)_+ } \frac{g \left( x, u_\lambda \right)}{u_\lambda} \left( \widetilde{u}_\lambda+a_- \right) \widetilde{u}_\lambda,
$$
where
$$
   \left( \mathbb R^N \setminus \Omega'_\Upsilon \right)_+ = \left\{ x \in \mathbb R^N \setminus \Omega'_\Upsilon \, ; \, u_\lambda(x) > 0 \right\},
$$
we derive
$$
   \int_{ \mathbb R^N \setminus \Omega'_\Upsilon } \! \! \big| \nabla \widetilde{u}_\lambda \big|^{ p(x) } + \int_{ \left( \mathbb R^N \setminus \Omega'_\Upsilon \right)_+ } \! \! \! \! \left( \big( \lambda V(x) + Z(x) \big) u_{\lambda}^{ p(x)-2 } -\frac{g \left( x, u_\lambda \right)}{u_\lambda} \right) \left( \widetilde{u}_\lambda+a_- \right)  \widetilde{u}_\lambda = 0,
$$
Now, by \eqref{til f estimate}, 
$$
   \big( \lambda V(x) + Z(x) \big) u_{\lambda}^{ p(x)-2 } - \frac{g \left( x, u_\lambda \right)}{u_\lambda} > \nu u_{\lambda}^{ p(x)-2 } - \frac{\tilde{f} \left( x, u_\lambda \right)}{u_\lambda} \ge 0 \quad \mbox{in} \quad  \left( \mathbb R^N \setminus \Omega'_\Upsilon \right)_+ .
$$
This form, $ \widetilde{u}_\lambda = 0 $ in $ \left( \mathbb R^N \setminus \Omega'_\Upsilon \right)_+ $. Obviously, $ \widetilde{u}_\lambda = 0  $ at the points where $ u_\lambda = 0 $, consequently, $ \widetilde{u}_\lambda = 0 $ in $ \mathbb R^N \setminus \Omega'_\Upsilon $. 		
\end{proof}


\section{A special critical value for $ \phi_\lambda $}


For each $ j = 1, \ldots, k $, consider
$$
   I_j(u) = \int_{ \Omega_j } \frac{1}{p(x)} \left( \big| \nabla u \big|^{ p(x) } + Z(x) | u |^{ p(x) } \right) - \int_{ \Omega_j } F(x,u), \ u \in W^{ 1,p(x) }_0 \big( \Omega_j \big),
$$
the energy functional associated to $ (P_j) $, and
$$
   \phi_{ \lambda,j }(u) = \int_{ \Omega'_j } \frac{1}{p(x)} \left( \big| \nabla u \big|^{ p(x) } + \big( \lambda V(x) + Z(x) \big) | u |^{ p(x) } \right) - \int_{ \Omega'_j } F(x,u), \ u \in W^{ 1,p(x) } \big( \Omega'_j \big),
$$
the energy functional associated to
$$
   \begin{cases}
	    - \Delta_{ p(x) } u + \big( \lambda V(x) + Z(x) \big) | u |^{ p(x)-2 } u  = f(x,u), \text{ in } \Omega'_j, \\
		  \phantom{ - \Delta_{ p(x) } u + \big( \lambda V(x) + Z(x) \big) | u |^{ p(x)-} } \frac{\partial u}{\partial \eta} = 0, \text{ on } \partial \Omega'_j.
	 \end{cases}
$$
It is fulfilled that $ I_j $ and $ \phi_{ \lambda,j } $ satisfy the mountain pass geometry and let
$$
   c_j = \inf_{ \gamma \in \Gamma_j } \max_{ t \in [0,1] } I_j \big( \gamma(t) \big) \, \text{ and } \, c_{ \lambda,j } = \inf_{ \gamma \in \Gamma_{ \lambda,j } } \max_{ t \in [0,1] } \phi_{ \lambda,j } \big( \gamma(t) \big),
$$
their respective mountain pass levels, where
$$
  \Gamma_j = \left\{ \gamma \in C \Big( [0,1], W^{ 1,p(x) }_0 \big( \Omega_j \big) \Big) \, ; \, \gamma(0) = 0 \text{ and } I_j \big( \gamma(1) \big) < 0 \right\}
$$
and
$$
	 \Gamma_{ \lambda,j } = \left\{ \gamma \in C \Big( [0,1], W^{ 1,p(x) } \big( \Omega'_j \big) \Big) \, ; \, \gamma(0) = 0 \text{ and } \phi_{ \lambda,j } \big( \gamma(1) \big) < 0 \right\}.
$$
Invoking the $ (PS) $ condition on $ I_j $ and $ \phi_{ \lambda,j } $, we ensure that there exist $ w_j \in W^{ 1,p(x) }_0 \big( \Omega_j \big) $ and $ w_{ \lambda,j } \in W^{ 1,p(x) } \big( \Omega'_j \big) $ such that
$$
   I_j \big( w_j \big) = c_j \, \text{ and } \, I'_j \big( w_j \big) = 0 
$$
and
$$
	\phi_{ \lambda,j } \big( w_{ \lambda,j } \big) = c_{ \lambda,j }	\, \text{ and } \, \phi'_{ \lambda,j } \big( w_{ \lambda,j } \big) = 0.
$$

\vspace{.3cm}
\begin{lemma}
   There holds that
\begin{enumerate}
   \item[(i)] $ 0 < c_{ \lambda,j } \le c_j, \, \forall \lambda \ge 1, \, \forall j \in \left\{ 1, \ldots, k \right\} $;
   \item[(ii)] $ c_{ \lambda,j } \to c_j, \text{ as } \lambda \to \infty, \, \forall j \in \left\{ 1, \ldots, k \right\} $.
\end{enumerate}
\end{lemma}

\begin{proof}
\begin{enumerate}
   \item[(i)]  Once $ W^{ 1,p(x) }_0 \big( \Omega_j \big) \subset W^{ 1,p(x) } \big( \Omega'_j \big) $ and $ \phi_{ \lambda,j } \big( \gamma(1) \big) = I_j \big( \gamma(1) \big) $ for $ \gamma \in \Gamma_j $, we have $ \Gamma_j \subset \Gamma_{ \lambda,j } $. This way
				 $$
				   c_{ \lambda,j } = \inf_{ \gamma \in \Gamma_{ \lambda,j } } \max_{ t \in [0,1] } \phi_{ \lambda,j } \big( \gamma(t) \big) \le \inf_{ \gamma \in \Gamma_j } \max_{ t \in [0,1] }            \phi_{ \lambda,j } \big( \gamma(t) \big) = \inf_{ \gamma \in \Gamma_j } \max_{ t \in [0,1] } I_j \big( \gamma(t) \big) = c_j.
				 $$
	 \item[(ii)] It suffices to show that $ c_{ \lambda_n,j } \to c_j, \text{ as } n \to \infty $, for all sequences $ ( \lambda_n ) $ in $ [1,\infty) $ with $ \lambda_n \to \infty, \text{ as } n \to \infty  $. Let $ \left( \lambda_n \right) $ be such a sequence and consider an arbitrary subsequence of $ \left( c_{ \lambda_n,j } \right) $ (not relabelled) . Let $ w_n \in W^{ 1,p(x) } \big( \Omega'_j \big) $ with 
		     $$
					  \phi_{ \lambda_n,j } \big( w_n \big) = c_{ \lambda_n,j }	\, \text{ and } \, \phi'_{ \lambda_n,j } \big( w_n  \big) = 0.
				 $$
By the previous item, $ \big( c_{ \lambda_n,j } \big) $ is bounded. Then, there exists $ \big( w_{ n_k } \big) $ subsequence of $ \big( w_n \big) $ such that $ \phi_{ \lambda_{ n_k },j } \big( w_{ n_k } \big) $ converges and $ \phi'_{ \lambda_{ n_k },j } \big( w_{ n_k }  \big) = 0 $. Now, repeating the same type of arguments explored in the proof of  Proposition \ref{(PS) infty condition}, there is $ w \in W^{ 1,p(x) }_0 \big( \Omega_j \big) \setminus \{0\} \subset W^{ 1,p(x) } \big( \Omega'_j \big) $ such that 
				 $$
				    w_{ n_k } \to w \text{ in } W^{ 1,p(x) } \big( \Omega'_j \big), \text{ as } k \to \infty.
				 $$
				 Furthermore, we also can prove that 
				 $$
				    c_{ \lambda_{ n_k },j } = \phi_{ \lambda_{ n_k },j } \big( w_{ n_k } \big) \to I_j(w)
				 $$
				and 
				 $$
				    0 = \phi'_{ \lambda_{ n_k },j } \big( w_{ n_k }  \big) \to I'_j(w).
				 $$
				Then,  by $ (f_4) $,
				 $$
				    \lim_k c_{ \lambda_{ n_k },j } \ge c_j.
				 $$
				 The last inequality together with item (i) implies
				 $$
				    c_{ \lambda_{ n_k },j } \to c_j, \text{ as } k \to \infty.
				 $$
				 This establishes the asserted result.
\end{enumerate}
\end{proof}

\vspace{.5cm}
In the sequel, let $ R > 1 $ verifying 
\begin{equation} \label{R}
   0< I_j \left( \frac{1}{R} w_j \right), I_j(R w_j)< c_j,  \text{ for } j = 1, \ldots, k.
\end{equation}
There holds that
$$
   c_j = \max_{ t \in [1/R^2,1] } I_j (t R w_j ), \text{ for } j = 1, \ldots, k.
$$
Moreover, to simplify the notation, we rename the components $ \Omega_j $ of $ \Omega $ in way such that $ \Upsilon = \{ 1, 2, \ldots, l \} $  for some $ 1 \le l \le k $. Then, we define:
\begin{gather*}
   \gamma_0 ( t_1, \ldots, t_l )(x) = \sum_{j=1}^l t_j R w_j(x), \, \forall ( t_1, \ldots, t_l )\in [1/R^2,1]^l, \\
   \Gamma_\ast = \Big\{ \gamma \in C \big( [1/R^2,1]^l, E_\lambda \setminus \{ 0 \} \big) \, ; \, \gamma = \gamma_0 \text{ on } \partial [1/R^2,1]^l \Big\}
\end{gather*}
and
$$
	 b_{ \lambda, \Upsilon } = \inf_{ \gamma \in \Gamma_\ast } \max_{ ( t_1, \ldots, t_l )\in [1/R^2,1]^l } \phi_\lambda \big( \gamma ( t_1, \ldots, t_l ) \big).
$$

\vspace{.3cm}
Next, our intention is proving that $ b_{ \lambda, \Upsilon } $ is a critical value for $ \phi_\lambda $. However, to do this, we need to some technical lemmas. The arguments used are the same found in \cite{Alves}, however for reader's convenience we will repeat their proofs

\begin{lemma} \label{solution's existence}
   For all $ \gamma \in \Gamma_\ast $, there exists $ (s_1, \ldots, s_l ) \in [1/R^2,1]^l $ such that
$$
   \phi'_{ \lambda,j } \big( \gamma ( s_1, \ldots, s_l ) \big) \big( \gamma ( s_1, \ldots, s_l ) \big) = 0, \, \forall j \in \Upsilon.
$$
\end{lemma}

\begin{proof}
   Given $ \gamma \in \Gamma_\ast $, consider $ \widetilde{\gamma} \colon [1/R^2,1]^l \to \mathbb R^l $ such that
$$
   \widetilde{\gamma} ( \textbf{t} ) = \Big( \phi'_{ \lambda,1 } \big( \gamma ( \textbf{t} ) \big) \gamma ( \textbf{t} ), \ldots, \phi'_{ \lambda,l } \big( \gamma ( \textbf{t} ) \big) \gamma ( \textbf{t} ) \Big), \text{ where } \textbf{t} = ( t_1, \ldots, t_l ).
$$
For $ \textbf{t} \in \partial [1/R^2,1]^l $, it holds $ \widetilde{\gamma} ( \textbf{t} ) = \widetilde{\gamma_0} ( \textbf{t} ) $. From this, we observe that there is no $ \textbf{t} \in \partial [1/R^2,1]^l  $ with $ \widetilde{\gamma} ( \textbf{t} ) = 0 $. Indeed, for any $ j \in \Upsilon $,
$$
   \phi'_{ \lambda,j } \big( \gamma_0 ( {\bf{t}} ) \big) \gamma_0 ( {\bf{t}} ) = I'_j ( t_j R w_j ) ( t_j R w_j ).
$$
This form, if $ {\bf{t}} \in \partial [1/R^2,1]^l $, then $ t_{j_0} =1 $ or $ t_{j_0} = \frac{1}{R^2} $, for some $ j_0 \in \Upsilon $. Consequently, 
$$
   \phi'_{ \lambda,j_0 } \big( \gamma_0 ( {\bf{t}} ) \big) \gamma_0 ( {\bf{t}} ) = I'_{j_0} ( R w_{j_0} ) ( R w_{j_0} ) \, \text{ or } \, \phi'_{ \lambda,j_0 } \big( \gamma_0 ( {\bf{t}} ) \big) \gamma_0 ( {\bf{t}} ) = I'_{j_0} \left( \frac{1}{R} w_{j_0} \right) \left( \frac{1}{R} w_{j_0} \right).
$$
Therefore, if $ \phi'_{ \lambda,j_0 } \big( \gamma_0 ( {\bf{t}} ) \big) \gamma_0 ( {\bf{t}} ) = 0 $, we get $ I_{j_0} ( R w_{j_0}  ) \ge c_{j_0} $ or $ I_{j_0} \left( \frac{1}{R} w_{j_0} \right) \ge c_{j_0} $, which is a contradiction with \eqref{R}.

Now, we compute the degree $ \deg \big( \widetilde{\gamma}, (1/R^2,1)^l, (0, \ldots, 0 ) \big) $. Since
$$
   \deg \big( \widetilde{\gamma}, (1/R^2,1)^l, (0, \ldots, 0 ) \big) = \deg \big( \widetilde{\gamma_0}, (1/R^2,1)^l, (0, \ldots, 0 ) \big),
$$
and, for $ \textbf{t} \in (1/R^2,1)^l $,
$$
  \widetilde{\gamma_0} ( \textbf{t} ) = 0 \iff {\bf{t}} = \left( \frac{1}{R}, \ldots, \frac{1}{R} \right),
$$
we derive
$$
   \deg \big( \widetilde{\gamma}, (1/R^2,1)^l, (0, \ldots, 0 ) \big) \ne 0.
$$
This shows what was stated.
\end{proof}

\begin{proposition} \label{blambdagamma}
   If $ c_{ \lambda,\Upsilon } = \displaystyle \sum_{ j=1 }^l c_{ \lambda,j } \, \text{ and } \, c_\Upsilon = \sum_{ j=1 }^l c_j $, then
\begin{enumerate}
   \item[(i)] $ c_{ \lambda,\Upsilon } \le b_{ \lambda,\Upsilon } \le c_\Upsilon, \, \forall \lambda \ge 1 $;
	 \item[(ii)] $ b_{ \lambda,\Upsilon } \to c_\Upsilon, \text{ as } \lambda \to \infty $;
	 \item[(iii)] $ \phi_\lambda \big( \gamma({\bf{t}}) \big) < c_\Upsilon, \, \forall \lambda \ge 1, \gamma \in \Gamma_\ast \text{ and } {\bf{t}} = (t_1, \ldots, t_l ) \in \partial [1/R^2,1]^l $.
\end{enumerate}	
\end{proposition}

\begin{proof}
\begin{enumerate}
   \item[(i)] Once $ \gamma_0 \in \Gamma_\ast $, 
	       $$
				    b_{ \lambda,\Upsilon } \le \max_{ ( t_1, \ldots, t_l ) \in [1/R^2,1]^l } \phi_\lambda \big( \gamma_0 ( t_1, \ldots, t_l ) \big) = \max_{ ( t_1, \ldots, t_l )\in [1/R^2,1]^l } \sum_{ j=1 }^l I_j ( t_j R w_j ) = c_\Upsilon.
				 $$
				 Now, fixing $ {\bf s} = (s_1, \ldots, s_l) \in [1/R^2,1]^l $ given in Lemma \ref{solution's existence} and recalling that
				 $$
				    c_{ \lambda,j } = \inf \left\{ \phi_{ \lambda,j } (u) \, ; \, u \in W^{ 1,p(x) } \big( \Omega'_j \big) \setminus \{ 0 \} \text{ and } \phi'_{ \lambda,j }(u)u = 0 \right\},
				 $$		
				 it follows that
				 $$
				    \phi_{ \lambda,j } \big( \gamma( {\bf s } ) \big) \ge c_{ \lambda,j }, \, \forall j \in \Upsilon.
				 $$
			From \eqref{til F estimate}, 
				 $$
				    \phi_{ \lambda, \mathbb R^N \setminus \Omega'_\Upsilon } (u) \ge 0, \, \forall u \in W^{ 1,p(x) } \big( \mathbb R^N \setminus \Omega'_\Upsilon \big),
				 $$
which leads to 
				 $$
				    \phi_\lambda \big( \gamma( {\bf t} ) \big) \ge \sum_{ j=1 }^l \phi_{ \lambda,j } \big( \gamma( {\bf t} ) \big), \, \forall \textbf{t} = (t_1, \ldots, t_l) \in [1/R^2,1]^l.
				 $$
				 Thus
				 $$
				    \max_{ ( t_1, \ldots, t_l )\in [1/R^2,1]^l } \phi_\lambda \big( \gamma ( t_1, \ldots, t_l ) \big) \ge \phi_\lambda \big( \gamma( \textbf{s} ) \big) \ge c_{ \lambda,\Upsilon },
				 $$
			showing that
				 $$
				    b_{ \lambda,\Upsilon } \ge c_{ \lambda,\Upsilon };
				 $$
	 \item[(ii)] This limit is clear by the previous item, since we already know $ c_{ \lambda,j } \to c_j $, as $ \lambda \to \infty $;
	 \item[(iii)] For $ \textbf{t} = ( t_1, \ldots, t_l ) \in \partial [1/R^2,1]^l $, it holds $ \gamma ( \textbf{t} ) = \gamma_0 ( \textbf{t} ) $. From this,
	       $$
				    \phi_\lambda \big( \gamma ( \textbf{t} ) \big) = \sum_{j=1}^l I_j ( t_j R w_j ).
				 $$
				 Writing
				 $$
				    \phi_\lambda \big( \gamma ( \textbf{t} ) \big) = \sum_{j=1 \atop j \ne j_0 }^l I_j ( t_j R w_j ) + I_{j_0} ( t_{j_0} R w_{j_0} ),
				 $$
				 where $ t_{j_0} \in \left\{ \frac{1}{R^2}, 1 \right\} $, from \eqref{R} we derive
				 $$
				    \phi_\lambda \big( \gamma ( \textbf{t} ) \big) \le c_\Upsilon - \epsilon,
				 $$
				 for some $ \epsilon > 0 $, so (iii).
\end{enumerate}
\end{proof}

\begin{corollary}
   $ b_{ \lambda,\Upsilon } $ is a critical value of $ \phi_\lambda $, for $ \lambda $ sufficiently large.	
\end{corollary}

\begin{proof}
   Assume $ b_{ \widetilde{\lambda},\Upsilon } $ is not a critical value of $ \phi_{\widetilde{\lambda}} $ for some $ \widetilde{\lambda}$. We will prove that exists $ \lambda_1 $ such that $ \widetilde{\lambda} < \lambda_1 $. Indeed, by item (iii) of Proposition \ref{blambdagamma}, we have seen that
$$
   \phi_\lambda \big( \gamma_0 ( \textbf{t} ) \big) < c_\Upsilon , \, \forall \lambda \ge 1, \, \textbf{t} \in \partial [1/R^2,1]^l.
$$
This way
$$
   {\cal M} = \max_{ \textbf{t} \in \partial [1/R^2,1]^l } \phi_{ \widetilde{\lambda} } \big( \gamma_0 ( \textbf{t} ) \big) < c_\Upsilon.
$$
Since $ b_{ \lambda,\Upsilon } \to c_\Upsilon $ (item (ii) of Proposition \ref{blambdagamma}), there exists $ \lambda_1 > 1 $ such that if $ \lambda \ge \lambda_1 $,  then
$$
   {\cal M} < b_{ \lambda,\Upsilon }.
$$
So, if $ \widetilde{\lambda} \ge \lambda_1 $, we can find $ \tau = \tau( \widetilde{\lambda} ) > 0 $ small enough, with the ensuing property 
\begin{equation} \label{M estimate}
   {\cal M} < b_{ \widetilde{\lambda},\Upsilon } - 2\tau.
\end{equation}
From the deformation's lemma \cite[Page 38]{W}, there is $ \eta \colon E_\lambda \to E_\lambda $ such that
$$
   \eta \left( \phi_{\widetilde{\lambda}}^{ b_{ \widetilde{\lambda},\Upsilon } +\tau } \right) \subset \phi_{\widetilde{\lambda}}^{ b_{ \widetilde{\lambda},\Upsilon} -\tau } \, \text{ and } \, \eta(u) = u, \text{ for } u \notin \phi_{\widetilde{\lambda}}^{-1} \big( [b_{ \widetilde{\lambda},\Upsilon }-2 \tau, b_{ \widetilde{\lambda},\Upsilon }+2 \tau] \big).
$$
Then, by \eqref{M estimate}, 
$$
   \eta \big( \gamma_0 ( \textbf{t} )\big) = \gamma_0 ( \textbf{t} ), \, \forall \textbf{t} \in \partial [1/R^2,1]^l.
$$
Now, using the definition of $ b_{ \widetilde{\lambda},\Upsilon } $, there exists $ \gamma_\ast \in \Gamma_\ast $ satisfying
\begin{equation} \label{estimate}
   \max_{ \textbf{t} \in [1/R^2,1]^l } \phi_{ \widetilde{\lambda} } \big( \gamma_\ast ( \textbf{t} ) \big) < b_{ \widetilde{\lambda},\Upsilon }+\tau.
\end{equation}
Defining
$$
   \widetilde{\gamma} ( \textbf{t} ) = \eta \big( \gamma_\ast ( \textbf{t} ) \big), \, \textbf{t} \in [1/R^2,1]^l,
$$
due to \eqref{estimate}, we obtain
$$
   \phi_{ \widetilde{\lambda} } \big( \widetilde{\gamma}( \textbf{t} ) \big) \le b_{ \widetilde{\lambda},\Upsilon } - \tau, \, \forall \textbf{t} \in [1/R^2,1]^l.
$$
But since $ \widetilde{\gamma} \in \Gamma_\ast $, we deduce
$$
   b_{ \widetilde{\lambda},\Upsilon } \le \max_{ \textbf{t} \in [1/R^2,1]^l } \phi_{\widetilde{\lambda}} \big( \widetilde{\gamma} ( \textbf{t} ) \big) \le b_{ \widetilde{\lambda},\Upsilon }-\tau,
$$
a contradiction. So, $ \widetilde{\lambda} < \lambda_1 $.
\end{proof}


\section{The proof of the main theorem}


To prove Theorem \ref{main}, we need to find nonnegative solutions $ u_\lambda $ for large values of $ \lambda $, which converges to a least energy solution in each $ \Omega_j $ $ (j \in \Upsilon) $ and to $ 0 $ in $\Omega_\Upsilon^{c}$ as $ \lambda \to \infty $. To this end, we will show two propositions which together with the Propositions \ref{(PS) infty condition} and  \ref{P:boundedness of the solutions} will imply that Theorem \ref{main} holds.

Henceforth, we denote by
$$
  r = R^{ p_+ } \sum_{ j=1 }^l \left( \frac{1}{p_+}-\frac{1}{\theta} \right)^{-1} c_j, \quad {\cal B}_r^\lambda = \big\{ u \in E_\lambda \, ; \, \varrho_\lambda (u) \le r \big\}
$$
and
$$
	\phi_\lambda^{ c_\Upsilon } = \big\{ u \in E_\lambda \, ; \, \phi_\lambda(u) \le c_{ \Upsilon } \big\}.
$$

Moreover, for small values of $ \mu $,
$$
   {\cal A}_\mu^\lambda = \left\{ u \in {\cal B}_r^{\lambda }\, ; \, \varrho_{ \lambda, \mathbb R^N \setminus \Omega_\Upsilon } (u) \le \mu, \, \left| \phi_{ \lambda,j }(u)-c_j \right| \le \mu, \, \forall j \in \Upsilon \right\}.
$$
We observe that
$$
   w = \sum_{ j=1 }^l w_j \in {\cal A}_\mu^\lambda \cap \phi_\lambda^{ c_\Upsilon },
$$
showing that $ {\cal A}_\mu^\lambda \cap \phi_\lambda^{ c_\Upsilon } \ne \emptyset $. Fixing
\begin{equation} \label{mu estimate}
   0 < \mu < \frac{1}{4} \min_{ j \in \Gamma } c_j,
\end{equation}
we have the following uniform estimate of $ \big\| \phi'_{ \lambda }(u) \big\| $ on the region $ \left( {\cal A}_{ 2 \mu }^\lambda \setminus {\cal A}_\mu^\lambda \right) \cap \phi_\lambda^{ c_\Upsilon } $.

\begin{proposition} \label{derivative estimate}
  Let $ \mu > 0 $ satisfying \eqref{mu estimate}. Then, there exist $ \Lambda_\ast \ge 1 $ and $ \sigma_0 >0   $ independent of $ \lambda $ such that
\begin{equation}
   \big\| \phi'_{ \lambda }(u) \big\| \ge \sigma_0, \text{ for } \lambda \ge \Lambda_\ast \text{ and all } u \in \left( {\cal A}_{ 2 \mu }^\lambda \setminus {\cal A}_\mu^\lambda \right) \cap \phi_\lambda^{ c_\Upsilon }.
\end{equation}
\end{proposition}

\begin{proof}
   We assume that there exist $ \lambda_n \to \infty $ and $ u_n \in \left( {\cal A}_{ 2 \mu }^{\lambda_n} \setminus {\cal A}_\mu^{\lambda_n} \right) \cap \phi_{\lambda_n}^{ c_\Upsilon } $ such that
$$
   \big\| \phi'_{ \lambda_n }(u_n) \big\| \to 0.
$$
Since $ u_n \in {\cal A}_{ 2 \mu }^{ \lambda_n } $, this implies $ \big( \varrho_{ \lambda_n } (u_n) \big) $ is a bounded sequence and, consequently, it follows that $ \big( \phi_{ \lambda_n }(u_n) \big) $ is also bounded. Thus, passing a subsequence if necessary, we can assume $\phi_{ \lambda_n }(u_n) $ converges. Thus, from Proposition \ref{(PS) infty condition}, there exists $ 0 \le u \in W^{ 1,p(x) }_0 \big( \Omega_\Upsilon \big) $ such that $ u_{ |_{ \Omega_j } }, \, j \in \Upsilon  $, is a solution for $ (P_j) $,
$$
	 \varrho_{ \lambda_n, \mathbb R^N \setminus \Omega_\Upsilon } (u_n) \to 0 \, \text{ and } \,  \phi_{ \lambda_n,j } (u_n) \to I_j(u).
$$
We know that $ c_j $ is the least energy level for $ I_j $. So, if $ u_{ |_{ \Omega_j } } \ne 0 $, then $ I_j(u) \ge c_j $. But since $ \phi_{ \lambda_n } (u_n) \le c_\Upsilon $, we must analyze the following possibilities:
\begin{enumerate}
   \item[(i)] $ I_j(u) = c_j, \, \forall j \in \Upsilon $;
	 \item[(ii)] $ I_{ j_0 }(u) = 0 $, for some $ j_o \in \Upsilon $.
\end{enumerate}

If (i) occurs, then for $ n $ large, it holds
$$
   \varrho_{ \lambda_n, \mathbb R^N \setminus \Omega_\Upsilon } (u_n) \le \mu \, \text{ and } \, \left| \phi_{ \lambda_n,j }(u_n)-c_j \right| \le \mu, \, \forall j \in \Upsilon.
$$
So $ u_n \in {\cal A}_\mu^{\lambda_n} $, a contradiction.

If (ii) occurs, then
$$
   \left| \phi_{ \lambda_n,j_0 }(u_n)-c_{ j_0 } \right| \to c_{ j_0 } > 4 \mu,
$$
which is a contradiction with the fact that $ u_n \in {\cal A}_{ 2 \mu }^{\lambda_n} $. Thus, we have completed the proof.
\end{proof}

\begin{proposition} \label{P}
   Let $ \mu > 0 $ satisfying \eqref{mu estimate} and $ \Lambda_\ast \ge 1 $ given in the previous proposition. Then, for $ \lambda \ge \Lambda_\ast $, there exists a solution $ u_\lambda $ of $ (A_\lambda) $ such that $ u_\lambda \in {\cal A}_\mu^\lambda \cap \phi_\lambda^{ c_\Upsilon } $.
\end{proposition}

\begin{proof}
  Let $ \lambda \ge \Lambda_\ast $. Assume that there are no critical points of $ \phi_\lambda $ in $ {\cal A}_\mu^\lambda \cap \phi_\lambda^{ c_\Upsilon } $. Since $ \phi_\lambda $ is a $ (PS) $ functional, there exists a constant $ d_\lambda > 0 $ such that
$$
   \big\| \phi'_\lambda(u) \big\| \ge d_\lambda, \text{ for all } u \in {\cal A}_\mu^\lambda \cap \phi_\lambda^{ c_\Upsilon }.
$$
From Proposition \ref{derivative estimate} we have
$$
   \big\| \phi'_\lambda(u) \big\| \ge \sigma_0, \text{ for all } u \in \left( {\cal A}_{ 2 \mu }^{\lambda} \setminus {\cal A}_\mu^{\lambda} \right) \cap \phi_{\lambda}^{ c_\Upsilon },
$$
where $ \sigma_0 > 0 $ does not depend on $ \lambda $. In what follows,  $ \Psi \colon E_\lambda \to \mathbb R $ is a continuous functional verifying 
$$
   \Psi(u) =  1, \text{ for } u \in {\cal A}_{\frac{3}{2} \mu}^\lambda, \ \Psi(u) = 0, \text{ for } u \notin {\cal A}_{2 \mu}^\lambda
\, \text{ and } \,  0 \le \Psi(u) \le 1, \, \forall u \in E_\lambda.
$$
We also consider $ H \colon \phi_\lambda^{ c_\Upsilon } \to E_\lambda $ given by
$$
   H(u) =
\begin{cases}
   - \Psi(u) \big\| Y(u) \big\|^{ -1 } Y(u), \text{ for } u \in {\cal A}_{2 \mu}^\lambda, \\
   \phantom{- \Psi(u) \big\| Y(u) \big\|^{ -1 } Y()} 0, \text{ for } u \notin {\cal A}_{2 \mu}^\lambda, \\
\end{cases}
$$
where $ Y $ is a pseudo-gradient vector field for $ \Phi_\lambda $ on $ {\cal K} = \left\{ u \in E_\lambda \, ; \, \phi'_\lambda(u) \ne 0 \right\} $. Observe that  $ H $ is well defined, once $ \phi'_\lambda(u) \ne 0 $, for $ u \in {\cal A}_{2 \mu}^\lambda \cap \phi_\lambda^{ c_\Upsilon } $. The inequality
$$
   \big\| H(u) \big\| \le 1, \, \forall \lambda \ge \Lambda_* \text{ and } u \in \phi_\lambda^{ c_\Upsilon },
$$
guarantees that the deformation flow $ \eta \colon [0, \infty) \times \phi_\lambda^{ c_\Upsilon } \to \phi_\lambda^{ c_\Upsilon } $ defined by
$$
   \frac{d \eta}{dt} = H(\eta), \ \eta(0,u) = u \in \phi_\lambda^{ c_\Upsilon }
$$
verifies
\begin{gather}
   \frac{d}{dt} \phi_\lambda \big( \eta(t,u) \big) \le - \frac{1}{2} \Psi \big( \eta(t,u) \big) \big\| \phi'_\lambda \big( \eta(t,u) \big) \big\| \le 0, \label{eta derivative}\\
	 \left\| \frac{d \eta}{dt} \right\|_\lambda = \big\| H(\eta) \big\|_\lambda \le 1
\end{gather}
and
\begin{equation} \label{eta}
   \eta(t,u) = u \text{ for all } t \ge 0 \text{ and } u \in \phi_\lambda^{ c_\Upsilon } \setminus {\cal A}_{2 \mu}^\lambda.
\end{equation}
We study now two paths, which are relevant for what follows: \\

$ \noindent \bullet $ The path $ {\bf t} \mapsto \eta \big( t, \gamma_0( {\bf t} ) \big), \text{ where } \textbf{t} = (t_1,\ldots,t_l) \in [1/R^2, 1]^l $.

\vspace{0.5 cm}

The definition of $ \gamma_0 $ combined with the condition on $ \mu $ gives 
$$
   \gamma_0( {\bf t } ) \notin {\cal A}_{2 \mu}^\lambda, \, \forall {\bf t } \in \partial [1/R^2, 1]^l.
$$
Since
$$
   \phi_\lambda \big( \gamma_0( {\bf t } ) \big) < c_\Upsilon, \, \forall {\bf t } \in \partial [1/R^2, 1]^l,
$$
from (\ref{eta}), it follows that
$$
   \eta \big( t, \gamma_0( {\bf t} ) \big) = \gamma_0( {\bf t} ), \, \forall {\bf t} \in \partial [1/R^2, 1]^l.
$$
So, $ \eta \big( t, \gamma_0( {\bf t} ) \big) \in \Gamma_\ast $, for each $ t \ge 0 $.

\vspace{0.5 cm}

$ \noindent \bullet $ The path $ {\bf t} \mapsto \gamma_0( {\bf t} ), \text{ where } \textbf{t} = (t_1,\ldots,t_l) \in [1/R^2, 1]^l $.

\vspace{0.5 cm}

We observe that
$$
   \text{supp} \big( \gamma_0 ( {\bf t} ) \big)\subset \overline{\Omega_\Upsilon}
$$
and
$$
   \phi_\lambda \big( \gamma_0 ( {\bf t} ) \big) \text{ does not depend on }  \lambda \ge 1,
$$
forall  $ {\bf t} \in [1/R^2, 1]^l $. Moreover,
$$
   \phi_\lambda \big( \gamma_0 ( {\bf t} ) \big) \le c_\Upsilon, \, \forall {\bf t} \in [1/R^2, 1]^l
$$
and
$$
   \phi_\lambda \big( \gamma_0 ( {\bf t} ) \big) = c_\Upsilon \text{ if, and only if, } t_j = \frac{1}{R}, \, \forall j \in \Upsilon.
$$
Therefore
$$
   m_0 = \sup \left\{ \phi_\lambda(u) \, ; \, u \in \gamma_0 \big( [1/R^2,1]^l \big) \setminus A_\mu^\lambda \right\}
$$
is independent of $ \lambda $ and $ m_0 < c_\Upsilon $. Now, observing that there exists $ K_\ast > 0 $ such that
$$
   \big| \phi_{ \lambda,j }(u) - \phi_{ \lambda,j }(v) \big| \le K_* \| u-v \|_{ \lambda, \Omega'_j }, \, \forall u,v \in {\cal B}_r^\lambda \text{ and } \forall j \in \Upsilon,
$$
we derive
\begin{equation} \label{max estimate}
    \max_{ {\bf t } \in [1/R^2,1]^l } \phi_\lambda \Big( \eta \big( T, \gamma_0 ( {\bf t} ) \big) \Big) \le \max \left\{ m_0, c_\Upsilon-\frac{1}{2 K_\ast} \sigma_0 \mu \right\},
\end{equation}
for $ T > 0 $ large.

In fact, writing $ u = \gamma_0( {\bf t} ) $, $ {\bf t } \in [1/R^2,1]^l $, if $ u \notin A_\mu^\lambda $, from (\ref{eta derivative}), 
$$
   \phi_\lambda \big( \eta( t, u ) \big) \le \phi_\lambda (u) \le m_0, \, \forall t \ge 0,
$$
and we have nothing more to do. We assume then $ u \in A_\mu^\lambda $ and set
$$
   \widetilde{\eta}(t) = \eta (t,u), \ \widetilde{d_\lambda} = \min \left\{ d_\lambda, \sigma_0 \right\} \text{ and } T = \frac{\sigma_0 \mu}{K_\ast \widetilde{d_\lambda}}.
$$
Now, we will analyze the ensuing cases: \\

\noindent {\bf Case 1:} $ \widetilde{\eta}(t) \in {\cal A}_{\frac{3}{2} \mu}^\lambda, \, \forall t \in [0,T] $.

\noindent {\bf Case 2:} $ \widetilde{\eta}(t_0) \in \partial {\cal A}_{\frac{3}{2} \mu}^\lambda, \text{ for some } t_0 \in [0,T] $. \\

\noindent {\bf Analysis of  Case 1}

In this case, we have $ \Psi \big( \widetilde{\eta}(t) \big) = 1 $ and $ \big\| \phi'_\lambda \big( \widetilde{\eta}(t) \big) \big\| \ge \widetilde{d_\lambda} $ for all $ t \in [0,T] $. Hence, from (\ref{eta derivative}),
$$
   \phi_\lambda \big( \widetilde{\eta}(T) \big) = \phi_\lambda (u) + \int_0^T \frac{d}{ds} \phi_\lambda \big( \widetilde{\eta}(s) \big) \, ds \le c_\Upsilon - \frac{1}{2} \int_0^T \widetilde{d_\lambda} \, ds,
$$
that is,
$$
   \phi_\lambda \big( \widetilde{\eta}(T) \big) \le c_\Upsilon - \frac{1}{2} \widetilde{d_\lambda} T = c_\Upsilon - \frac{1}{2 K_\ast} \sigma_0 \mu,
$$
showing (\ref{max estimate}). \\

\noindent {\bf Analysis of Case 2}

In this case, there exist $ 0 \le t_1 \le t_2 \le T $ satisfying
\begin{gather*}
   \widetilde{\eta}(t_1) \in \partial {\cal A}_\mu^\lambda, \\
	 \widetilde{\eta}(t_2) \in \partial {\cal A}_{\frac{3}{2} \mu}^\lambda,
\end{gather*}
and
$$
   \widetilde{\eta}(t) \in {\cal A}_{\frac{3}{2} \mu}^\lambda \setminus {\cal A}_\mu^\lambda, \, \forall t \in (t_1,t_2].
$$
We claim that 
$$
   \big\| \widetilde{\eta}(t_2)-\widetilde{\eta}(t_1) \big\| \ge \frac{1}{2 K_\ast} \mu.
$$
Setting $ w_1 = \widetilde{\eta}(t_1) $ and $ w_2 = \widetilde{\eta}(t_2) $, we get
$$
   \varrho_{ \lambda, \mathbb R^N \setminus \Omega_\Upsilon } (w_2) = \frac{3}{2} \mu \ \text{ or } \, \big| \phi_{ \lambda, j_0 } (w_2) - c_{j_0} \big| = \frac{3}{2} \mu,
$$
for some $ j_0 \in \Upsilon $.  We analyse the latter situation, once that the other one follows the same reasoning. From the definition of $ {\cal A}_\mu^\lambda $,
$$
   \big| \phi_{ \lambda, j_0 } (w_1) - c_{j_0} \big| \le \mu,
$$
consequently,
$$
   \| w_2-w_1 \| \ge \frac{1}{K_\ast} \big| \phi_{ \lambda, j_0 } (w_2) - \phi_{ \lambda, j_0 } (w_1) \big| \ge \frac{1}{2 K_\ast} \mu.
$$
Then, by  mean value theorem, $ t_2-t_1 \ge \frac{1}{2 K_\ast} \mu $ and, this form,
$$
   \phi_\lambda \big( \widetilde{\eta}(T) \big) \le \phi_\lambda(u) - \int_0^T \Psi \big( \widetilde{\eta}(s) \big) \big\| \phi'_\lambda \big( \widetilde{\eta}(s) \big) \big\| \, ds
$$
implying
$$
   \phi_\lambda \big( \widetilde{\eta}(T) \big) \le c_\Upsilon - \int_{t_1}^{t_2} \sigma_0 \, ds = c_\Upsilon - \sigma_0 (t_2-t_1) \le c_\Upsilon - \frac{1}{2 K_\ast} \sigma_0 \mu,
$$
which proves \ref{max estimate}. Fixing $ \widehat{\eta} (t_1, \ldots, t_l) = \eta \big( T, \gamma_0 (t_1,\ldots,t_l) \big) $, we have that $ \widehat{\eta} \in \Gamma_\ast $ and, hence,
$$
   b_{ \lambda, \Gamma } \le \max_{ (t_1,\ldots,t_l) \in [1/R^2, 1] } \phi_\lambda \big( \widehat{\eta} (t_1,\ldots,t_l) \big) \le \max \left\{ m_0, c_\Upsilon - \frac{1}{2 K_\ast} \sigma_0 \mu \right\} < c_\Upsilon,
$$
which contradicts the fact that $ b_{ \lambda, \Upsilon } \to c_\Upsilon $.
\end{proof}

\vspace{.5cm}
\begin{proof} [Proof of Theorem \ref{main}]
According Proposition \ref{P}, for $\mu$ satisfying \eqref{mu estimate} and $ \Lambda_\ast \ge 1 $, there exists a solution $ u_\lambda $ for $ (A_\lambda) $ such that $ u_\lambda \in {\cal A}_\mu^\lambda \cap \phi_\lambda^{ c_\Upsilon } $, for all $\lambda \geq \Lambda_*$. \\

\noindent {\bf Claim:} 
There are $\lambda_0 \geq \Lambda_*$ and $\mu_0>0$ small enough, such that $u_\lambda$ is a solution for $ \big( P_\lambda \big)$ for $\lambda \geq \Lambda_0$ and $\mu \in (0, \mu_0)$.  

Indeed, assume by contradiction that there are  $ \lambda_n \to \infty $ and $ \mu_n \to 0 $, such that $(u_{\lambda_n})$ is not a solution for $(P_{\lambda_n})$. From Proposition \ref{P}, the sequence $ (u_{\lambda_n}) $ verifies:
\begin{enumerate}
   \item[(a)] $ \phi'_{ \lambda_n }(u_{\lambda_n}) = 0, \, \forall n \in \Bbb N $;
	 \item[(b)] $ \varrho_{ \lambda_n,  \Bbb R^N \setminus \Omega_\Upsilon }(u_{\lambda_n}) \to 0$;  
	 \item[(c)] $ \phi_{ \lambda_n,j } (u_{\lambda_n}) \to c_j, \, \forall j \in \Upsilon. $
\end{enumerate}	
The item (b) ensures we can use Proposition \ref{P:boundedness of the solutions} to deduce $ u_{\lambda_n} $ is a solution for $ \big( P_{\lambda_n} \big) $, for large values of $ n $, which is a contradiction, showing this way the claim. \\

Now, our goal is to prove the second part of the theorem. To this end, let  $(u_{\lambda_n})$ be a sequence verifying the above limits. Since $ \phi_{ \lambda_n }(u_{ \lambda_n } ) $ is bounded, passing a subsequence, we obtain that $ \phi_{ \lambda_n }(u_{ \lambda_n } ) \to c $. This way, using Proposition \ref{(PS) infty condition} combined with item (c), we derive $ u_{ \lambda_n } $ converges in $ W^{ 1,p(x) } \big( \Bbb R^N \big ) $ to a function $ u \in W^{ 1,p(x) } \big( \Bbb R^N \big ) $, which satisfies $ u = 0 $ outside $ \Omega_\Upsilon $ and $ u_{|_{\Omega_j}}, \, j \in \Upsilon $, is a least energy solution for
$$
   \begin{cases}
		  - \Delta_{ p(x) } u + Z(x) u  = f(u), \text{ in } \Omega_j, \\
		  u \in W^{ 1,p(x) }_0 \big( \Omega_j \big), \, u \ge 0, \text{ in } \Omega_j.
	 \end{cases}
$$
\end{proof}


\begin{thebibliography}{99}
   \bibitem{Acerbi1} E. Acerbi \& G. Mingione, Regularity results for stationary electrorheological fluids, {\it Arch. Rational Mech. Anal.} {\bf 164} (2002), 213-259.
   \bibitem{Acerbi2} E. Acerbi \& G. Mingione, Regularity results for electrorheological fluids: stationary case, {\it C.R. Math. Acad. Sci. Paris} {\bf 334} (2002), 817-822.
	 \bibitem{Alves} C.O. Alves, Existence of multi-bump solutions for a class of quasilinear problems, {\it Adv. Nonlinear Stud.} {\bf 6} (2006), 491-509.
	 \bibitem{Alves2} C.O. Alves, Existence of solutions for a degenerate $ p(x) $-Laplacian equation in $ \mathbb R^N $, {\it J. Math. Anal. Appl.} {\bf 345} (2008) 731-742.
   \bibitem{Alves3} C.O. Alves, Existence of radial solutions for a class of $ p(x) $-Laplacian equations with critical growth, {\it Differential and Integral Equations} {\bf 23} (2010), 113-123.
	 \bibitem{AlvesBarreiro} C.O. Alves \& J.L.P. Barreiro, Existence and multiplicity of solutions for a $ p(x) $-Laplacian equation with critical growth, {\it J. Math. Anal. Appl.} {\bf 403} (2013) 143-154.
    \bibitem{AlvesFerreira} C.O. Alves \& M.C. Ferreira, Nonlinear perturbations of a $p(x)$-Laplacian equation with critical growth in $ \mathbb R^N $, to appear in {\it Math. Nach.} (2013).
   \bibitem{AlvesFerreira1} C.O. Alves \& M.C. Ferreira, Existence of solutions for a class of $ p(x)$-Laplacian equations involving a concave-convex nonlinearity with critical growth in $ \mathbb R^N $, to appear in {\it Topol. Methods Nonlinear Anal.} (2013).
	 \bibitem{AlvesSouto}  C.O. Alves \& M.A.S. Souto, Existence of solutions for a class of problems in $ \mathbb R^N $ involving $ p(x) $-Laplacian, {\it Prog. Nonlinear Differential Equations and their Appl.} {\bf 66 } (2005), 17-32.
   \bibitem{AmRb} A. Ambrosetti \& P.H. Rabinowitz, Dual variational methods in critical point theory and applications, J. Funct. Anal. 14 (1973) 349-381.  
	 \bibitem{Antontsev} S.N. Antontsev \& J.F. Rodrigues, On stationary thermo-rheological viscous flows, {\it Ann. Univ. Ferrara Sez. VII Sci. Mat.}  {\bf 52} (2006), 19-36.
   \bibitem{CLions} A. Chambolle \& P.L. Lions, Image recovery via total variation minimization and related problems, {\it Numer. Math.} {\bf 76} (1997), 167-188.
   \bibitem{Chen} Y. Chen, S. Levine \& M. Rao,  Variable exponent, linear growth functionals in image restoration, {\it SIAM J. Appl. Math.} { \bf 66} (2006), 1383-1406.
   \bibitem{DelPinoFelmer} M. del Pino \& P.L. Felmer, Local mountain passes for semilinear elliptic problems in unbounded domains, { \it Calc. Var. PDE} {\bf 4} (1996), 121-137.
	 \bibitem{DingTanaka} Y.H. Ding \& K. Tanaka, Multiplicity of positive solutions of a nonlinear Schr\"odinger equation, {\it Manuscripta Math.} {\bf 112(1)} (2003) 109-135
	 \bibitem{Fan1} X.L. Fan, On the sub-supersolution method for $ p(x) $-Laplacian equations, {\it J. Math. Anal. Appl.} {\bf 330} (2007), 665-682.
	 \bibitem{Fan} X.L. Fan, $ p(x) $-Laplacian equations in $ \mathbb R^N $ with periodic data and nonperiodic perturbations, {\it J. Math. Anal. Appl.} {\bf 341} (2008), 103-119.
	 \bibitem{FanZhao0} X. Fan \& D. Zhao, A class of De Giorgi type and H\"older continuity, {\it Nonlinear Anal.} {\bf 36} (1999), 295-318.
	 \bibitem{FZ} X.L. Fan \& D. Zhao, On the Spaces $ L^{ p(x) } \big( \Omega \big) $ and $ W^{ 1,p(x) } \big( \Omega \big) $, {\it J. Math. Anal. Appl.} {\bf 263} (2001), 424-446.
	 \bibitem{FZ1} X.L. Fan \& D. Zhao, Nodal solutions of $ p(x) $-Laplacian equations, {\it Nonlinear Anal.} {\bf 67} (2007), 2859-2868.
	 \bibitem{FanShenZhao} X.L. Fan, J.S. Shen \&  D. Zhao, Sobolev embedding theorems for spaces $ W^{ k,p(x) } \big( \Omega \big) $, {\it J. Math. Anal. Appl.} {\bf 262} (2001) 749-760.
	 \bibitem{BSS} J. Fern\'andez Bonder, N. Saintier \& A. Silva. On the Sobolev embedding theorem for variable exponent spaces in the critical range, {\it J. Differential Equations} {\bf 253} (2012), 1604-1620
   \bibitem{FuZa} Y. Fu \& X. Zhang, Multiple solutions for a class of $ p(x) $-Laplacian equations in involving the critical exponent, {\it Proceedings Roy. Soc. of Edinburgh Sect A}  {\bf 466} (2010) 1667-1686.
   \bibitem{FuscoSbordone} N. Fusco \& C. Sbordone, Some remarks on the regularity of minima of anisotropic integrals, {\it Comm. Partial Differential Equations } {\bf 18(1-2)} (1993), 153-167.
	 \bibitem{Kavian} O. Kavian, Introduction \`a la th\'eorie de points critiques et applications aux probl\`emes elliptiques, Springer-Verlag France, Paris, 1993.
   \bibitem{LadyUral} O. A. Ladyzhenskaya \& N. N. Ural'tseva, Linear and quasilinear elliptic equations, Acad. Press, 1968.
	 \bibitem{MR} M. Mih$\breve{a}$ilescu \& V. R$\breve{a}$dulescu, On a nonhomogeneous quasilinear eigenvalue problem in Sobolev spaces with variable exponent, {\it Proc. Amer. Math. Soc.} {\bf 135(9)} (2007) 2929-2937 (electronic).	
   \bibitem{Ru} M. Ruzicka, Electrorheological  fluids: Modeling and mathematical theory. Lecture Notes in Mathematics, vol. 1748, Springer-Verlag, Berlin, 2000.
	 \bibitem{Sere} E. S\'er\'e, Existence of infinitely many homoclinic orbits in Halmitonian systems, {\it Math Z} {\bf 209} (1992), 27-42.
	 \bibitem{W} M. Willem, Minimax Theorems, Birkh\"auser Boston, MA, 1996.

\end{thebibliography}
\end{document}